\documentclass[12pt, a4paper, reqno]{amsart}
\usepackage{graphicx}
\usepackage{amssymb, amsthm, xcolor}
\usepackage{mathrsfs}%花文字スタイルファイルmathscr
\usepackage{braket}%ブラケットの表示<x>
\usepackage[centering, heightrounded]{geometry}%12ptのほうがよい

\usepackage{enumitem}%enumerate余白設定 enumitem.sty
\setlist[enumerate]{topsep=3pt, parsep=0pt, partopsep=0pt, itemsep=3pt, leftmargin=36pt, label=\rm{(\roman*)}}
\numberwithin{equation}{section}

\newtheorem{theorem}{Theorem}[section]
\newtheorem{lemma}[theorem]{Lemma}
\newtheorem{proposition}[theorem]{Proposition}

\theoremstyle{definition}

\theoremstyle{remark}
\newtheorem{remark}[theorem]{Remark}

\newtheorem*{remark*}{Remark}

\renewcommand{\r}{\right}
\newcommand{\eps}{\varepsilon}

\newcommand{\R}{{\mathbb R}}
\newcommand{\C}{{\mathbb C}}

\renewcommand{\Re}{\operatorname{Re}}
\renewcommand{\Im}{\operatorname{Im}}

\newcommand{\pt}{\partial}
\newcommand{\cleq}{\lesssim}
\newcommand{\cgeq}{\gtrsim}
\newcommand{\wto}{\rightharpoonup}%弱収束

\def\tbra[#1,#2]{\left\langle #1 , #2\right\rangle} %inner product%trianguler bracket <・>
\def\rbra[#1,#2]{\left( #1 , #2 \right)} %inner product% round bracket　(・)

\newcommand{\ce}{\mathrel{\mathop:}=}

\def\norm[#1]{\left\Vert #1 \right\Vert}
\def\abs[#1]{\left\vert #1 \right\vert}

\newcommand{\supp}{\operatorname{supp}}
\newcommand{\sgn}{\operatorname{sgn}}

\newcommand{\scF}{{\mathscr F}}

\newcommand{\cE}{{\mathcal E}}

\title[]{The Cauchy problem for the logarithmic Schr\"odinger equation revisited}
\author[M. Hayashi]{Masayuki Hayashi}
\address{Dipartimento di Matematica, Universit\`a di Pisa, Largo Bruno Pontecorvo, 5 56127 Pisa, Italy
\newline\indent
Waseda Research Institute for Science and Engineering, Waseda University, Tokyo 169-8555, Japan
}
 \email{masayuki.hayashi@dm.unipi.it}
 
 \author[T. Ozawa]{Tohru Ozawa}
\address{Department of Applied Physics, Waseda University, Tokyo 169-8555, Japan}
\email{txozawa@waseda.jp}

\begin{document}

\maketitle
\setcounter{tocdepth}{1}%{1}sectionまで表示

\begin{abstract}
We revisit the Cauchy problem for the logarithmic Schr\"odinger equation and construct strong solutions in $H^1$, the energy space, and the $H^2$-energy space. The solutions are provided in a constructive way, which does not rely on compactness arguments, that a sequence of approximate solutions forms a Cauchy sequence in a complete function space and then actual convergence is shown to be in a strong sense.
%	the strong convergence is obtained.  
\end{abstract}

\tableofcontents

\section{Introduction}

We consider the logarithmic Schr\"odinger equation
%$\cprime$
\begin{align}
\label{eq:1.1}
\left\{
\begin{aligned}
&i\pt_t u+\Delta u+\lambda u\log (|u|^2)=0,
%~\lambda\in\R\setminus\{0\},
\\
& u(0,x)=\varphi(x),
%u_{|_{t=0}}=\varphi,
%u(0)=\varphi,
%\quad\text{where}~\lambda\in\R\setminus\{0\}.
\end{aligned}
\quad(t,x)\in\R\times\R^d,~\lambda\in\R\setminus\{0\}.
\r.
\end{align}
%where $\lambda\in\R\setminus\{0\}$. 
%I.Białynicki-Birula
%\begin{itemize}
%\item Gausson stationary \cite{BM76, BM79, dMS14} stability \cite{C83, A16}
%\item multi-Gausson, multi-Breathers \cite{F21}
%\item numerical \cite{BCST19a, BCST19b} 
%\item superposition \cite{BCST19b, F20}
%\item Cauchy problem \cite{CH80, C03, CG18, H18}
%\item survey \cite{C22-survey}
%\item generalized DCT \cite{Fo99}
%\item physics literature \cite{BM76, Z10}
%\item Cauchy problem add \cite{GLN10}
%\end{itemize}
This model was first introduced in \cite{BM76} and later found to be suitable for describing various physical phenomena \cite{Hef85, KEB00, Z10, AZ11} (mainly the case $\lambda>0$ is of physical interests).  
The study of the Cauchy problem for \eqref{eq:1.1} goes back to \cite{CH80} and is discussed in detail later. It is known that the properties of the solution of \eqref{eq:1.1} differ considerably depending on the sign of $\lambda$. For instance, on one hand, when $\lambda>0$ the equation has a non-dispersive structure \cite{C83} and explicit standing waves called Gaussons \cite{BM79, dMS14}, which are shown to be stable \cite{C83, A16}. Multi-Gaussons and multi-breathers are also studied numerically and rigorously \cite{BCST19b, F20, F21}. On the other hand, when $\lambda<0$ it is shown in \cite{CG18} that \eqref{eq:1.1} has an interesting dispersive structure that is very different from nonlinear Schr\"odinger equation with standard power nonlinearities. 
%We refer to \cite{C22-survey} for 
More complete references can be found in the recent survey \cite{C22-survey}. 

In this paper we revisit the Cauchy problem for \eqref{eq:1.1}. The energy for this equation is given by
\begin{align}
\label{eq:1.2}%
\cE (u)=\frac{1}{2}\int|\nabla u|^2 -\frac{\lambda}{2}\int |u|^2\log(|u|^2).
\end{align}
The integrand with logarithmic functions has an indeterminate sign, which changes at $|u|=1$. If we rewrite
%If we rewrite the last integral in %eqref{eq:1.2}% as
\begin{align*}
-\frac{\lambda}{2}\int |u|^2\log(|u|^2)=-\frac{\lambda}{2}\left(\int_{|u|\le1}+\int_{|u|\ge1} \r)|u|^2\log(|u|^2),
\end{align*}
then the tricky term is 
%難しい部分は原点の特異性をもつ部分
\begin{align}
\label{eq:1.3}
-\frac{\lambda}{2}\int_{|u|\le1}|u|^2\log(|u|^2),
\end{align}
which cannot be controlled only for $u\in H^1(\R^d)$. 
%%
%%%
%
The main difficulty in the Cauchy problem for \eqref{eq:1.1} is that the nonlinear term has a singularity at the origin and breaks the local Lipschitz continuity.
% For example, a rough evaluation of a nonlinear term yields
%If we roughly estimate the nonlinear term, for any $\delta\in(0,1)$
On one hand, the nonlinear term is roughly estimated as 
\begin{align}
\label{eq:1.4}
\abs[u\log(|u|^2) ]\cleq  |u|^{1-\delta}+|u|^{1+\delta} 
%\quad\text{for}~u\in\C.
\end{align}
for any $\delta\in(0,1)$, and the problem arises that if $u\in H^1(\R^d)$, the first term on the right-hand side of \eqref{eq:1.4} does not in general belong to any $L^p(\R^d)$ for $p\le 2$, nor to $H^{-1}(\R^d)$.
%the first term on the right-hand side of \eqref{eq:1.4} cannot be controlled in the Sobolev space $H^s(\R^d)$ for any $s\in\R$.
%%
On the other hand, the inequality
\begin{align}
\label{eq:1.5}
\bigl| \Im ( u\log |u|-v\log |v| )(\overline{u}-\overline{v})\bigr|
\le |u-v|^2\quad\text{for all}~u,v\in\C
\end{align}
is known to hold (see \cite[Lemme 1.1.1]{CH80}) and this helps overcome the lack of the local Lipschitz continuity of the nonlinear term.
%this is important for avoiding the singularity at the origin of the nonlinear term.

The energy \eqref{eq:1.2} is well-defined in the subset of $H^1(\R^d)$
\begin{align}
\label{eq:1.6}
W_1\ce\left\{ \varphi\in H^1(\R^d) :|\varphi|^2\log(|\varphi|^2)\in L^1(\R^d) \r\},
\end{align}
which is known to be represented as a Banach space using Orlicz spaces (see \cite{C83} and Appendix \ref{sec:B}). 
The Cauchy problem for \eqref{eq:1.1} in the energy space $W_1$
was first investigated in \cite{CH80} and it was proved that if $\lambda>0$ and $u_0\in W_1$, then there exists a unique solution $u\in C(\R, W_1)$ to \eqref{eq:1.1} (see \cite{H18} for an alternative proof). In \cite{GLN10} the authors studied \eqref{eq:1.1} including the case $\lambda<0$ and they proved that if $\varphi\in H^1(\R^3)\cap\scF(H^{1/2})$, where 
\begin{align*}
\scF(H^\alpha)\ce\left\{  \varphi\in L^2(\R^d) : x\mapsto\braket{x}^\alpha\varphi(x)\in L^2(\R^d) \r\}
%\quad\text{for}~\alpha>0~\text{and}~\braket{x}\ce (1+|x|^2)^{1/2},
\end{align*} 
for $\alpha>0$ and $\braket{x}\ce (1+|x|^2)^{1/2}$, then there exists a unique solution in $L^\infty_{\rm loc}(\R,H^1(\R^3))$.
%cap C(\R,L^2(\R^3) )
Later, this existence result in weighted Sobolev spaces is improved in \cite{CG18} and it was proved that if $\varphi\in H^m(\R^d)\cap\scF(H^\alpha)$ with $m=1,2$ and $\alpha\in(0,1]$, then there exists a unique solution $u\in L^\infty_{\rm loc}(\R, H^m(\R^d)\cap\scF(H^\alpha) )$. The space $H^1(\R^d)\cap\scF(H^\alpha)$ enables us to control both \eqref{eq:1.3} and the first term on the right-hand side of \eqref{eq:1.4}. Note, however, that this space is strictly narrower than the energy space (see \eqref{eq:1.8}).

Regarding the construction of solutions to \eqref{eq:1.1}, in most previous works compactness arguments are used probably due to the singularity of the nonlinear term. 
%In this paper we revisit the Cauchy problem for \eqref{eq:1.1} 
Here, as with \cite{H18}, we construct solutions in a more constructive way that does not rely on compactness arguments. Specifically, we construct solutions by showing that a sequence of approximate solutions forms a Cauchy sequence in a complete function space. The advantage of this approach is that the strong convergence can be obtained directly without going through the weak convergence and without taking subsequences.
%%%%%
%The differences from \cite{H18} and the novelty of this paper are as follows:
The novelty of this paper is as follows:
\vspace{3pt}
\begin{itemize}
\setlength{\itemsep}{3pt}

\item We adopt the approximation of the nonlinear term used in \cite{BCST19a} to simplify the argument of \cite{H18} and treat the case $\lambda<0$ as well as $\lambda>0$.

\item The strong $H^1$-solution is constructed in a way that does not rely on the conservation law of the energy. 

\item The strong $W_1$-solution is newly constructed for the case $\lambda<0$.

%\item The Cauchy problem for the case $\lambda<0$ as well as $\lambda>0$ is treated in a unified way.
%%
%\footnote{In \cite{CH80, C03, H18} the approximation, which cuts off the neighborhood of each the origin and the infinity in the nonlinear term, is used to decompose the nonlinear term according to the Orlicz space.}

%\item The solution in the energy space is constructed without using the properties of the Orlicz space.

\item The strong solution in the $H^2$-energy space is newly constructed for the case $\lambda>0$.

%We construct solutions not only in the usual energy space but also in the $H^2$-energy space.
\end{itemize}
\vspace{3pt}
%%%%%%%%%%%
The $H^2$-energy space is defined by
\begin{align}
\label{eq:1.7}
W_2&=\left\{ \varphi\in H^2(\R^d) : \varphi\log(|\varphi|^2)\in L^2(\R^d) \r\},
\end{align}
which can be represented as a Banach space using the Orlicz space in the same way as $W_1$. 
%%%
%Although the Cauchy problem for \eqref{eq:1.1} is a classic topic, our results include new ones. Our main contributions in this paper are the construction of solutions in the energy space when $\lambda<0$ and in the $H^2$-energy space when $\lambda>0$.
%construction of solutions 
%%%
%Function spaces
%%%%
We note that the inclusion relation in function spaces
\begin{align}
\label{eq:1.8}
\begin{aligned}
&H^1(\R^d)\cap\scF(H^\alpha)\subsetneq W_1\subsetneq H^1(\R^d),
\\
&H^2(\R^d)\cap\scF(H^\alpha)\subsetneq W_2\subsetneq H^2(\R^d)
\end{aligned}
\end{align}
holds for $\alpha\in(0,1]$. One can see from Lemma \ref{lem:B.5} below that the difference in the logarithmic decay order of the functions distinguishes these function spaces. This seemingly slight difference is not a small issue in \eqref{eq:1.1}.
%where the strictness of the inclusion is proved in Appendix \ref{sec:B}.
%空間遠方の$\log$オーダーの減衰の違いが函数空間の違いを与えていることに注意する．

We now state our main results. The first main result is on the construction of strong solutions in $H^1$ and $W_1$.
\begin{theorem}
\label{thm:1.1}
%\leavevmode\\
%\indent
Let $\lambda\in\R\setminus\{0\}$. For any $\varphi\in H^1(\R^d)$, there exists a unique solution $C(\R, H^1(\R^d))$
to \eqref{eq:1.1} in the sense of
\begin{align}
\label{eq:1.9}
i\pt_t u +\Delta u+\lambda u\log(|u|^2)=0\quad\text{in}~H^{-1}(\Omega)
%\quad\text{for all}~\Omega\subset\subset\R^d.
\end{align}
for all bounded open sets $\Omega\subset\R^d$ and all $t\in\R$, and with $u(0)=\varphi$.
If in addition we assume $\varphi\in W_1$, the $H^1$-solution above satisfies $u\in (C\cap L^\infty)(\R,W_1)$ if $\lambda>0$ and $u\in C(\R,W_1)$ if $\lambda<0$. Moreover, the $W_1$-solution $u$ satisfies the equation \eqref{eq:1.9} in the sense of $W_1^*$, where $W_1^*$ is the dual space of $W_1$.
\end{theorem}
%%%%%%%%%%%%%%%%%
%%%%%%%%%%%%%%%%%
The strong $H^1$-solution to \eqref{eq:1.1} can be constructed for the initial data $\varphi\in H^1(\R^d)$ (not assuming $\varphi\in W_1$), which seems to have not been pointed out in previous works. We use the idea of \cite[Remarks (c)]{KL84} to construct strong solutions in $H^1$ and $W_1$. More specifically,  the lower semicontinuity with respect to the time variable is obtained by the construction of weak solutions, and the upper semicontinuity is obtained by energy inequalities. This argument can also be applied to  regularity issues in $W_2$ (but more delicate). 
%The argument of construction of solutions is done in a almost unified way for the cases $\lambda>0$ and $\lambda<0$.
%%%%%%%%%%%%%%%%%%%
%%%%%%%%%%%%%%%%%%%

The second main result is on the construction of strong solutions in $W_2$.
\begin{theorem}
\label{thm:1.2}
Let $\lambda\in\R\setminus\{0\}$. For any $\varphi\in W_2$, there exists a unique solution 
%of \eqref{eq:1.1} in the class of 
\begin{gather*}
u\in W^{1,\infty}_{\rm loc}(\R, L^2(\R^d)),~
\pt_t u\in C_w(\R, L^2_{\rm loc}(\R^d)),
\\
\Delta u \in (C_w\cap L^\infty_{\rm loc})(\R, L^2_{\rm loc}(\R^d)),
%\quad \forall \Omega\subset\subset\R^d
\end{gather*}
to \eqref{eq:1.1} in the sense of
\begin{align}
\label{eq:1.10}
i\pt_t u +\Delta u+\lambda u\log(|u|^2)=0\quad\text{in}~L^2(\Omega)
%\quad\text{for all}~\Omega\subset\subset\R^d.
\end{align}
for all bounded open sets $\Omega\subset\R^d$ and a.e. $t\in\R$, with $u(0)=\varphi$.
Moreover, when $\lambda>0$, $u\in C(\R, W_2)$ and \eqref{eq:1.10} holds in $L^2(\R^d)$ and for all $t\in\R$.
%holds with $\Omega=\R^d$
%by replacing $L^2(\Omega)$ with $L^2(\R^d)$.
\end{theorem}
%%%%%%%%%%%%%%
%%%%%%%%%%%%%%
%$W_2$における解の構成で鍵となるのは$H^2$-エネルギーに関する等式
The key to constructing the solution in $W_2$ is the $H^2$-identity
\begin{align}
\label{eq:1.11}%
\norm[\pt_t u]_{L^2}^2=
\begin{aligned}[t]
&\norm[\Delta u]_{L^2}^2-4\lambda\norm[\nabla |u|]_{L^2}^2
\\
&{}~-2\lambda\rbra[\abs[\nabla u]^2,\log(|u|^2)]_{L^2}
%-2\lambda\int \abs[\nabla u]^2\log|u|^2
+\lambda^2\norm[u\log(|u|^2)]_{L^2}^2,
\end{aligned}
\end{align}
which is formally obtained by the equation \eqref{eq:1.1} and integration by parts.  
%%%
As in \cite{CG18}, the left-hand side of \eqref{eq:1.11} can be controlled by the energy inequality of time derivatives. The difficulty in constructing solutions in $W_2$ is that we need to obtain a priori estimates for $\norm[\Delta u]_{L^2}$ and $\norm[u\log(|u|^2)]_{L^2}$ at the same time.
%eqref{eq:1.8}%で最も困るのは右辺第三項であり，特に
If we rewrite the third term on the right-hand side of \eqref{eq:1.11} as
\begin{align*}
-\lambda\rbra[\abs[\nabla u]^2,\log(|u|^2)]_{L^2}=-\lambda\left(\int_{|u|\le 1}+\int_{|u|\ge 1}\r)|\nabla u|^2\log (|u|^2), 
\end{align*}
the most tricky term is the integral term near the origin. However, when $\lambda>0$, this term can be dropped as a non-negative term, which enables us to obtain the desired a priori estimates.  
%%%
It is more difficult to improve the regularity of the solution from $L^\infty_{\rm loc}(\R, W_2)$ to $C(\R,W_2)$. We take advantage of \eqref{eq:1.11} again and it is solved by combining limiting procedures and energy inequalities. 
%%%%%%%%%

In constructing solutions in $W_1$ and $W_2$, the sign of $\lambda$ has a sensitive effect on the analysis. We note that this is not seen in the construction of solutions in the weighted Sobolev spaces as in \cite{GLN10, CG18}.

%\cite[Theorem 9.3.4]{C03}
The rest of the paper is organized as follows. In Section \ref{sec:2} we study the Cauchy problem for \eqref{eq:1.1} in $H^1$ and $W_1$, and prove Theorem \ref{thm:1.1}. In Section \ref{sec:3} we study the Cauchy problem in $W_2$ and prove Theorem \ref{thm:1.2}. As with \cite{C03, H18}, our main results still hold for a general domain. For the convenience of the reader, in Section \ref{sec:4} we restate Theorems \ref{thm:1.1} and \ref{thm:1.2} for the case of a general domain.
%\\[10pt]
\subsection*{Notation}
We often write $\int f$ instead of $\int_{\R^d} f(x)dx$. 
The standard scalar product in $L^2(\R^d)=L^2(\R^d,\C)$ is defined by
\begin{align*}
\rbra[f,g]_{L^2}=\int_{\R^d} f(x)\overline{g(x)}dx=\int f\overline{g}.
\end{align*}
%and we keep this notation whenever $f\overline{g}\in L^1(\R^d, \C)$.
We denote by $C(I,X)$ (resp. $C_w(I,X)$) the space of strongly (resp. weakly) continuous functions from $I$ to $X$, where $I\subset\R$ is an interval and $X$ is a Banach space or a Fr\'echet space. We sometimes use the abbreviated notation such as
\begin{align*}
C_T(X)=C([-T,T] ,X),\quad L^\infty_T(X)=L^\infty((-T,T) ,X)\quad\text{for}~T>0.
\end{align*}
We denote by $2^*$ the Sobolev exponent defined by
\begin{align*}
2^*=\left\{
\begin{aligned}
&\frac{2d}{d-2}&&\text{if}~d\ge3,
\\
&\,\infty&&\text{if}~d=1,2.
\end{aligned}
\r.
\end{align*}
We denote by $B_R$ the open ball of radius $R$ with center at the origin of $\R^d$. For open sets $\omega, \Omega\subset\R^d$, we write $\omega\subset\subset\Omega$ if $\bar{\omega}\subset\Omega$ and $\bar{\omega}$ is compact, where $\bar{\omega}$ is the closure of $\omega$ in $\R^d$.

We use $A\cleq B$ to denote the inequality $A\le CB$ for some constant $C>0$. The dependence of $C$ 
%on the parameters 
is usually clear from the context and we often omit this dependence. We sometimes denote by $C=C(*)$ a constant depending on the quantities appearing in parentheses to clarify the dependence.

\section{The Cauchy problem in $H^1$ and the energy space}
%The Cauchy problem in the energy space
\label{sec:2}

%In this section we study the Cauchy problem for \eqref{eq:1.1} in $H^1$ and $W_1$. 
This section is organized as follows. In Section \ref{sec:2.1} we introduce the approximate equation of \eqref{eq:1.1} and the basic properties of approximate solutions. In Section \ref{sec:2.2} we construct weak $H^1$-solutions and study the uniqueness and regularity in Section \ref{sec:2.3}. 
%The key step is to prove that approximate solutions form a Cauchy sequence in $C([-T,T], L^2_{\rm loc}(\R^d))$ for any $T>0$. 
%$C_{\rm loc}(\R, L^2_{\rm loc})$
In Section \ref{sec:2.4} we construct solutions in the energy space $W_1$.

\subsection{Approximate problems}
\label{sec:2.1}
We consider the approximate equation
\begin{align}
\label{eq:2.1}
\left\{
\begin{aligned}
&i\pt_t u_\eps+\Delta u_\eps+2\lambda u_\eps\log (|u_\eps|+\eps)=0,
\\
&
u_\eps(0, x)=\varphi(x),
\end{aligned}
\r.
\quad(t,x)\in\R\times\R^d,~\eps>0.
\end{align}
We set 
\begin{align*}
g(u)= 2u\log|u|,\quad g_\eps(u)=2u\log(|u|+\eps)\quad\text{for}~\eps>0. 
\end{align*}
For $s\ge0$ we have
\begin{align*}
\int_0^s g_\eps(\tau)d\tau 
&=\frac{1}{2}s^2\log \left((s+\eps)^2\r)-\frac{1}{2}
\int_0^s\frac{2\tau^2}{\tau+\eps} d\tau.
%\underbrace{ }_{\sim s^2~(\eps\to0)}.
\end{align*}
We define $G_\eps(u)$ by
\begin{align*}
G_\eps(u)&=\frac{1}{2}\int |u|^2\log\left((|u|+\eps)^2\r)-\frac{1}{2}\int \mu_\eps(|u|)\quad\text{for}~u\in H^1(\R^d),
%\quad
%\mu_\eps(s)\ce\int_0^s \frac{2\tau^2}{\tau+\eps}d\tau\quad\text{for}~s\ge0
\end{align*}
where
\begin{align*}
\mu_\eps(s)\ce\int_0^s \frac{2\tau^2}{\tau+\eps}d\tau\quad\text{for}~s\ge0.
\end{align*}
We define $E_\eps(u)$ by 
%%%%%%
\begin{align}
\begin{aligned}
 E_\eps(u)&=\frac{1}{2}\int|\nabla u|^2-\lambda G_\eps(u)
 \\
&=\frac{1}{2}\int|\nabla u|^2-\frac{\lambda}{2}\int|u|^2\log\left( (|u|+\eps)^2\r)+\frac{\lambda}{2}\int \mu_\eps(|u|),
\end{aligned}
\label{eq:2.2}
\end{align}
which corresponds to the energy of \eqref{eq:2.1}. It follows from the well-posedness theory (see \cite[Chapter 4]{C03}) that for any $\varphi\in H^1(\R^d)$ there exists a unique solution 
%in the class of 
\begin{align*}
u_\eps \in C(\R, H^1(\R^d))\cap C^1(\R, H^{-1}(\R^d))
\end{align*}
to \eqref{eq:2.1}. Moreover, the solution satisfies the conservation laws of the $L^2$-norm
\begin{align}
\label{eq:2.3}
\norm[u_\eps(t)]_{L^2}^2=\norm[\varphi]_{L^2}^2
% \quad\text{for all}~t\in\R
\end{align}
and of the energy
\begin{align}
\label{eq:2.4}
E_\eps(u_\eps(t))=E_\eps(\varphi)
%\quad\text{for all}~t\in\R.
\end{align}
for all $t\in\R$.
%%%%%%%%%
\begin{remark}
\label{rem:2.1}
%We note that the nonlinearity of \eqref{eq:2.1} has no sigularity of the origin and its growth is almost linear. 
If we consider \eqref{eq:2.1} in a general domain, the well-posedness of \eqref{eq:2.1} in $H^1$ is proved by applying \cite[Theorem 3.3.9]{C03}, although the compactness arguments are used therein. Alternatively one can also prove this well-posedness of \eqref{eq:2.1} independent of compactness arguments by applying the argument in \cite[Section 2.1]{H18}. Although power type nonlinearities in two dimensions are discussed therein, the argument is applicable to \eqref{eq:2.1} for any dimension thanks to the almost linear growth of $g_\eps$ and the estimate \eqref{eq:2.8}.
\end{remark}
%%%%
\begin{remark}
\label{rem:2.2}
The energy $E_\eps(u)$ with $\eps=0$ is given by
\begin{align}
\label{eq:2.5}
E(u)=\frac{1}{2}\int|\nabla u|^2-\frac{\lambda}{2}\int|u|^2\log(|u|^2)+\frac{\lambda}{2}\int |u|^2.
\end{align}
We note that \eqref{eq:1.1} is rewritten as the Hamiltonian form $i\pt_t u=E'(u)$, so from this viewpoint $E(u)$ may be more suitable as the energy of \eqref{eq:1.1} than $\cE(u)$.
\end{remark}
%%%%
\begin{remark}
\label{rem:2.3}
As an approximation for the nonlinear term, the authors in \cite{CG18} adopt $g_\eps(u)=u\log(|u|^2+\eps)$. Our small modification of this approximation is helpful to generalize the inequality \eqref{eq:1.5} to hold for the approximate nonlinearity (see \eqref{eq:2.8}).
\end{remark}

\subsection{Construction of weak $H^1$-solutions}
\label{sec:2.2}
We first derive the uniform estimate of approximate solutions in $H^1$.
\begin{lemma}
\label{lem:2.4}
For all $t\in\R$ we have
\begin{align}
\label{eq:2.6}
\norm[\nabla u_\eps(t)]_{L^2}^2 
\le e^{4\abs[\lambda]\abs[t]}\norm[\nabla\varphi]_{L^2}^2.
\end{align}
\end{lemma}
%%%%%%%%%%
%%%%%%%%%%
\begin{proof}
%方程式\eqref{eq:2.1}と標準的なエネルギー不等式の計算により示せる．
This is proved by the standard energy inequality. Indeed, we have
\begin{align*}
\frac{d}{dt}\norm[\nabla u_\eps(t)]_{L^2}^2
&=2\Re\rbra[\pt_t\nabla u_\eps,\nabla u_\eps]
\\
&=2\Im\rbra[i\pt_tu,-\Delta u_\eps]
\\
&=4\lambda\Im\rbra[u_\eps\log(|u_\eps|+\eps),\Delta u_\eps]
=-4\lambda\Im\rbra[\frac{u_\eps}{|u_\eps|+\eps}\nabla|u_\eps|,\nabla u_\eps].
\end{align*}
%(see \cite[Theorem 6.17]{LL})
From the fact $\abs[\nabla |u_\eps|]\le |\nabla u_\eps|$ (see also \eqref{eq:3.10}), we obtain
\begin{align*}
\frac{d}{dt}\norm[\nabla u_\eps(t)]_{L^2}^2\le 4|\lambda|\norm[\nabla u_\eps]_{L^2}^2\quad\text{for all}~t\in\R.
\end{align*}
Then the result follows from Gronwall's lemma.
\end{proof}
%%%%%%
It follows from \eqref{eq:2.3} and \eqref{eq:2.6} that for any $T>0$ we have
\begin{align}
\label{eq:2.7}
M_T\ce\sup_{0<\eps<1}\norm[u_\eps]_{C_T(H^1)}\le C(T,\norm[\varphi]_{H^1}).
\end{align}
%where we used the abbreviation
%\begin{align*}
%C_T(H^1)=C_T(H^1(\R^d)) =C([-T,T], H^1(\R^d)).
%\end{align*}
%%%%%%
Next we prove that $\{u_\eps\}_{0<\eps<1}$ forms a Cauchy sequence in $C_T(L^2_{\rm loc}(\R^d))$ as $\eps\downarrow0$ for any $T>0$.
The inequality  
\begin{align}
\label{eq:2.8}
\bigl|\Im \bigl(u\log (|u|+\eps)-v\log(|v|+\mu) \bigr) (\overline{u}-\overline{v})\bigr|
\le |u-v|^2+\abs[\eps-\mu]\abs[u-v],
\end{align}
which is regarded as a generalization of \eqref{eq:1.5}, is essential in our proof. The proof of \eqref{eq:2.8} is given in Appendix \ref{sec:A}.
%%%%%%%%%
%%%%%%%%%
%We now prepare cut-off functions. 

Take a function $\zeta\in C^{\infty}_c(\R^d)$ satisfying
\begin{align*}
\zeta (x)=
\left\{
\begin{aligned}
&1&&\text{if}~|x|\le 1,
\\
&0 &&\text{if}~|x|\ge 2,
\end{aligned}
\r.
\qquad 0\le\zeta (x)\le 1\quad\text{for all}~x\in\R^d.
\end{align*}
For $R>0$ we set $\zeta_R\ce \zeta(x/R)$.
%\begin{align}
%%label{eq:2.9}%
%\zeta_R\ce \zeta(x/R),~B_R\ce\{x\in\R^d : |x|\le R \}.
%\end{align}
%%
For $\eps,\mu\in(0,1)$ we proceed with the calculation of the quantity $\norm[\zeta_R(u_\eps-u_\mu)]_{L^2}^2$. 
A direct calculation, based on \eqref{eq:2.1}, \eqref{eq:2.7}, and \eqref{eq:2.8}, shows that
\begin{align*}
\frac{d}{dt}\norm[\zeta_R(u_\eps-u_\mu)]_{L^2}^2
&=2\Im\rbra[i\zeta_R^2\pt_t(u_\eps-u_\mu),u_\eps-u_\mu]
\\
&=
\begin{aligned}[t]
&2\Im\rbra[\nabla(\zeta_R^2)\nabla(u_\eps-u_\mu),u_\eps-u_\mu]
\\
&{}+4\lambda\Im\rbra[\zeta_R^2
\left(u_\eps\log (|u_\eps|+\eps)-u_\mu\log(|u_\mu|+\mu)\r),
u_\eps-u_\mu]
\end{aligned}
\\
&\le 
\begin{aligned}[t]
&\frac{C}{R}\norm[\nabla(u_\eps-u_\mu)]_{L^2}\norm[u_\eps-u_\mu]_{L^2}
\\
&{}+4\abs[\lambda]\left(\norm[\zeta_R(u_\eps-u_\mu)]_{L^2}^2+\abs[\eps-\mu]\norm[\zeta_R^2(u_\eps-u_\mu)]_{L^1} \r).
\end{aligned}
\end{align*}
Integrating the both sides of the last inequality over $[0,t]$ and applying Gronwall's lemma, we obtain that
\begin{align}
\label{eq:2.9}
\norm[\zeta_R(u_\eps-u_\mu)(t)]_{L^2}^2
\le e^{4|\lambda|T}T\left( \frac{C(M_T)}{R}+\abs[\eps-\mu]\abs[B_{2R}]^{1/2}\norm[\varphi]_{L^2}\r)
%\quad\text{for all}~t\in [-T,T],
\end{align}
for all $t\in[-T,T]$, where we have used 
\begin{align*}
\norm[\zeta_R^2(u_\eps-u_\mu)]_{L^1}\le\norm[u_\eps-u_\mu]_{L^2(B_{2R})}\le2\abs[B_{2R}]^{1/2}\norm[\varphi]_{L^2}.
\end{align*}
We now fix $R_0>0$ and take $R\in(R_0,\infty)$ as a parameter. It follows from \eqref{eq:2.9} that 
\begin{align*}
\norm[u_\eps-u_\mu]_{C_T(L^2(B_{R_0}) )}^2
\le \norm[\zeta_R(u_\eps-u_\mu)]_{C_T(L^2)}^2
\le C(T,\norm[\varphi]_{H^1})\left(\frac{1}{R}+\abs[\eps-\mu]\abs[B_{2R}]^{1/2}\r),
\end{align*}
which yields that
\begin{align*}
\limsup_{\eps,\mu\downarrow0}\norm[u_\eps-u_\mu]_{C_T(L^2(B_{R_0}) )}^2\le \frac{C(T,\norm[\varphi]_{H^1})}{R}
\underset{R\to\infty}{\longrightarrow} 0.
\end{align*}
%We note that the left-hand side is independent of $R$ and the right-hand side goes to zero as $R\to\infty$. 
Since $R_0>0$ is arbitrary, we deduce that $\{u_\eps\}_{0<\eps<1}$ forms a Cauchy sequence of $C_T(L^2_{\rm loc}(\R^d))$. Combining this with \eqref{eq:2.3}, we deduce that there exists $u\in L^{\infty}(\R,L^2(\R^d))$ such that
\begin{align}
\label{eq:2.10}
u_\eps \to u \quad\text{in}~C_T(L^2_{\rm loc}(\R^d) )\quad
\text{as $\eps\downarrow0$}
\end{align}
for all $T>0$.
%%%
%Combining this with \eqref{eq:2.7}, we obtain the following.
\begin{lemma}
\label{lem:2.5}
$u\in L^\infty_{\rm loc}(\R, H^1(\R^d))$ and 
\begin{align}
\label{eq:2.11}
u_\eps(t)\wto u(t) \quad\text{in}~H^1(\R^d)\quad\text{for all}~t\in\R.
%~\forall t\in\R.
\end{align}
\end{lemma}
%%%%%%
\begin{proof}
First it follows from \eqref{eq:2.10} that 
\begin{align}
\label{eq:2.12}
u_\eps(t)\wto u(t) \quad\text{in}~L^2(\R^d)\quad\text{for all}~t\in\R.
\end{align}
To prove $u\in L^\infty_{\rm loc}(\R, H^1(\R^d))$, we use the characterization of $H^1$ functions. For any $\psi\in C^1_c(\R^d)$ and $t\in [-T,T]$, we obtain from \eqref{eq:2.7} that
\begin{align*}
\abs[\int u_\eps(t)\nabla\psi]=\abs[\int \nabla u_\eps(t)\psi]
\le M_T\norm[\nabla\psi]_{L^2}.
\end{align*}
Then, it follows from \eqref{eq:2.12} that
\begin{align*}
\abs[\int u(t)\nabla\psi]\le M_T\norm[\nabla\psi]_{L^2}
\quad\text{for all}~t\in[-T,T].
\end{align*}
Applying \cite[Proposition 9.3]{Bre11}, we deduce that for each $t\in[-T,T]$
\begin{align*}
u(t)\in H^1(\R^d)\quad \text{and}\quad\norm[\nabla u(t)]_{L^2}\le M_T,
\end{align*}
which yields that $u\in L^\infty_{\rm loc}(\R, H^1(\R^d))$.
Moreover, it follows from \eqref{eq:2.12} that 
\begin{align*}
\int\nabla u_\eps(t)\psi \to
\int \nabla u(t)\psi
%\quad\text{for all}~t\in\R.
\end{align*}
for any $\psi\in C^1_c(\R^d)$ and all $t\in\R$. Using \eqref{eq:2.7} and a density argument, we deduce that
\begin{align}
\label{eq:2.13}
\nabla u_\eps(t)\wto \nabla u(t) \quad\text{in}~L^2(\R^d)
%\quad\text{for all}~t\in\R.
\end{align}
for all $t\in\R$. This completes the proof.
\end{proof} 
%%%%%%%%%%%%%

For the nonlinear term we obtain from \eqref{eq:1.4} that $g(u(t))\in L^2_{\rm loc}(\R^d)$.
%for all $t\in\R$. 
Regarding convergence, we prove the following.
\begin{lemma}
\label{lem:2.6}
For all $t\in\R$ we have
\begin{align*}
g_\eps(u_\eps(t))\to g(u(t))~\text{in}~L^2_{\rm loc}(\R^d)
\quad\text{as}~\eps\downarrow0.
\end{align*}
%$g_\eps(u_\eps(t))\to g(u(t))~\text{in}~L^2_{\rm loc}(\R^d)$ 
\end{lemma}
%%%%%%%%%
%%%%%%%%%
\begin{proof}
We show that for any $\Omega\subset\subset\R^d$ and $t\in\R$,
\begin{align}
\label{eq:2.14}
u_\eps(t)\log  (|u_\eps(t)|+\eps) \to u(t)\log |u(t)|\quad\text{in}~L^2(\Omega)\quad\text{as}~\eps\downarrow0 .
\end{align}
It follows from \eqref{eq:A.2} that for any small $\delta>0$ there exists $C(\delta)>0$ such that
\begin{align*}
%%label{log-ineq}%
\bigl| {u_\eps\log(|u_\eps|+\eps)-u\log\abs[u]} \bigr|
\le  
\eps+|u_\eps-u|
+C(\delta) (1+|u_\eps|^{\frac{1}{2}+\delta}+|u|^{\frac{1}{2}+\delta})\abs[u_\eps-u]^{1/2}.
\end{align*}
We only check the most nontrivial estimate on the right-hand side. Fix $\delta>0$ satisfying $2+4\delta<2^*$. Then, we have
\begin{align*}
\norm[ {|u_\eps|^{\frac{1}{2}+\delta}\abs[u_\eps-u]^{1/2} } ]_{L^2(\Omega)}^2
&=\int_{\Omega} |u_\eps|^{1+2\delta}\abs[u_\eps-u]
\\
&\le \left(\int_\Omega|u_\eps|^{2+4\delta}\r)^{1/2}\norm[u_\eps-u]_{L^2(\Omega)}
\\
&\le C(\norm[u_\eps]_{H^1})\norm[u_\eps-u]_{L^2(\Omega)}.
%\underset{\eps\to0}{\longrightarrow}0.
\end{align*}
Therefore the result follows from \eqref{eq:2.7} and \eqref{eq:2.10}.
\end{proof}
%%%%%%%%%
%%%%%%%%%%

We deduce from the equation \eqref{eq:2.1} that for every $\psi\in C^{\infty}_c(\R^d)$ and every $\phi\in C^1_c(\R)$,
\begin{align*}
\int_\R \rbra[i u_\eps,\psi]_{L^2}\phi'(t)dt
&=-\int_\R \tbra[i\pt_t u_\eps,\psi]_{H^{-1},H^1}\phi(t)dt
\\
&=\int_\R \tbra[\Delta u_\eps+2\lambda u_\eps\log(|u_\eps|+\eps),\psi]_{H^{-1},H^1}\phi(t)dt
\\
&=\int_\R\bigl\{ -\rbra[\nabla u_\eps,\nabla\psi]_{L^2}+\rbra[\lambda g_\eps(u_\eps),\psi]_{L^2} \bigr\} \phi(t)dt.
\end{align*}
Passing to the limit as $\eps\downarrow0$ and we obtain from \eqref{eq:2.11} and Lemma~\ref{lem:2.6} that
\begin{align*}
%%label{eq:1.13}%
\int_\R \rbra[i u,\psi]_{L^2}\phi'(t)dt
=\int_\R\bigl\{ -\rbra[\nabla u,\nabla\psi]_{L^2}+\rbra[\lambda g(u),\psi]_{L^2} \bigr\} \phi(t)dt.
\end{align*}
It is easily verified from this formula and \eqref{eq:2.11} that for any $\Omega\subset\subset\R^d$,
\begin{align*}
u\in L^{\infty}_{\rm loc}(\R,H^1(\R^d))\cap W^{1,\infty}_{\rm loc}(\R, H^{-1}(\Omega))
\end{align*}
and
\begin{align}
\label{eq:2.15}
i\pt_t u+\Delta u+\lambda g(u)=0\quad\text{in}~H^{-1}(\Omega)
\end{align}
for a.e. $t\in\R$.
%%%%%%%%%%%%
\begin{remark}
\label{rem:2.7}
The proof of Lemma \ref{lem:2.5} is proved in a way that does not rely on weak compactness. The advantage of this argument is that we do not need to take subsequences.
\end{remark}
%%%%%%%%%%%
%%%%%%%%%%%
\begin{remark}
\label{rem:2.8}
In the proof of \cite{CH80, C03}, the convergence of the nonlinear term $g_\eps(u_\eps)$ is proved by using some compactness result as in \cite[Lemme 1.2]{JL69}, \cite[Theorem 1.2]{S70}, \cite[Proposition 1.2.1]{C03}. Our approach does not require any compactness argument at this step either. We note that the proof of Lemma~\ref{lem:2.6} depends on elementary logarithmic inequalities.
%elementary inequality of the logarithmic function
\end{remark}
%%%%%%%%%%%
%%%%%%%%%%%

\subsection{Uniqueness and regularity}
\label{sec:2.3}
Uniqueness of solutions follows from the following lemma. 
\begin{lemma}[{\cite[Lemme 2.2.1]{CH80}}]
\label{lem:2.9}
Assume that $u,v\in L^\infty_T(H^1(\R^d))$ for some $T>0$ satisfies \eqref{eq:1.1} in the distribution sense. Then, $u=v$.
\end{lemma}
%%%%%%%%
\begin{proof}
For completeness we give a proof. The argument runs parallel to the above proof of that $\{u_\eps\}_{0<\eps<1}$ forms a Cauchy sequence. We set
\begin{align*}
M\ce \max\left\{ \norm[u]_{L^\infty_T(H^1)},  \norm[v]_{L^\infty_T(H^1)}\r\}.
\end{align*}
As mentioned above, $u,v$ satisfies the equation in the sense of \eqref{eq:2.15}. Then, using the cutoff function $\zeta_R$, we obtain from \eqref{eq:1.5} that
\begin{align*}
\frac{d}{dt}\norm[\zeta_R(u-v)]_{L^2}^2
&=2\Im\tbra[i\zeta_R^2\pt_t(u-v),u-v]_{H^{-1}(B_{2R}), H^1_0(B_{2R})}
\\
&=
\begin{aligned}[t]
&2\Im\rbra[\nabla(\zeta_R^2)\nabla(u-v),u-v]_{L^2}
\\
&{}+4\lambda\Im\rbra[\zeta_R^2
\left(u\log |u|-v\log|v|\r),
u-v]_{L^2}
\end{aligned}
\\
&\le 
\frac{C(M)}{R}+4\abs[\lambda]\norm[\zeta_R(u-v)]_{L^2}^2.
\end{align*}
Applying Gronwall's lemma, 
\begin{align*}
\norm[\zeta_R(u-v)(t)]_{L^2}^2\le e^{4|\lambda|T} \frac{C(M)}{R}\quad
\text{for all}~t\in [-T,T].
\end{align*}
Applying Fatou's lemma, 
\begin{align*}
\norm[(u-v)(t)]_{L^2}^2
\le\liminf_{R\to\infty}\norm[\zeta_R(u-v)(t)]_{L^2}^2\le0
%\quad\forall t\in (-T,T).
\end{align*}
for all $t\in[-T,T]$. This yields that $u=v$ on $[-T,T]$.
\end{proof}
%%%%%%%%%%
%As a continuation of Section \ref{sec:2.2}, we first show the following.
Combining Lemma \ref{lem:2.9}, one can improve the regularity of the solution.
\begin{lemma}
\label{lem:2.10}
$u\in C_w(\R, H^1(\R^d))\cap C(\R,L^2(\R^d))$ and 
\begin{align*}
u_\eps (t) \to u(t)\quad \text{in}~L^2(\R^d)
\end{align*}
for all $t\in\R$.
\end{lemma}
%%%%
\begin{proof}
First we note that 
$
u\in C_w(\R, H^1(\R^d))$.
Indeed this easily follows from Lemma \ref{lem:2.5} and 
$u\in C(\R,L^2_{\rm loc}(\R^d))$. 
Next, we obtain from \eqref{eq:2.3} and \eqref{eq:2.12} that
\begin{align*}
\norm[u(t)]_{L^2}^2\le\liminf_{\eps\to0}\norm[u_\eps(t)]_{L^2}^2=\norm[\varphi]_{L^2}^2\quad\text{for all}~t\in\R.
\end{align*}
%Combined with Lemma~\ref{lem:2.9}, this
Uniqueness of solutions yields that
\begin{align}
\label{eq:2.16}
\norm[u(t)]_{L^2}^2=\norm[\varphi]_{L^2}^2
\quad\text{for all}~t\in\R,
\end{align}
see the proof of \cite[Theorem 3.3.9]{C03} for more details. Combining this with $u\in C_w(\R,L^2(\R^d))$, we deduce that 
$u\in C(\R,L^2(\R^d))$. Moreover, the claim of the convergence follows from \eqref{eq:2.12}, \eqref{eq:2.3}, and \eqref{eq:2.16}.
\end{proof}
%%%%%%%%%%%%%%
\begin{lemma}
\label{lem:2.11}
$u\in C(\R, H^1(\R^d))$.
\end{lemma}
\begin{proof}
%We use the idea of \cite[Remarks (c)]{KL84}. 
It is enough to show the continuity of $t\mapsto u(t) \in H^1(\R^d)$ at $t=0$. It follows from \eqref{eq:2.6}, \eqref{eq:2.11}, and the weak lower semicontinuity of the norm that
\begin{align*}
\norm[\nabla u(t)]_{L^2}^2 
\le e^{4\abs[\lambda]\abs[t]}\norm[\nabla\varphi]_{L^2}^2.
\end{align*}
Passing to the limit as $t\to0$ we have
\begin{align}
\label{eq:2.17}
\limsup_{t\to0}\norm[\nabla u(t)]_{L^2}^2 \le \norm[\nabla\varphi]_{L^2}^2.
\end{align}
On the other hand, it follows from the weak continuity $t\mapsto u(t) \in H^1(\R^d)$ at $t=0$ that 
\begin{align*}
\norm[\nabla\varphi]_{L^2}^2\le\liminf_{t\to0}\norm[\nabla u(t)]_{L^2}^2.
\end{align*}
Combined with \eqref{eq:2.17}, this yields that
\begin{align*}
\lim_{t\to0}\norm[\nabla u(t)]_{L^2}^2 =\norm[\nabla\varphi]_{L^2}^2.
\end{align*}
Since $\nabla u(t)\wto \nabla\varphi$ weakly in $L^2(\R^d)$ as $t\to0$, we deduce that
\begin{align*}
\nabla u(t)\to \nabla\varphi\quad\text{in}~L^2(\R^d).
\end{align*}
This completes the proof.
\end{proof}
This completes the proof of the first part of Theorem \ref{thm:1.1}.
\begin{remark}
\label{rem:2.12}
In \cite{CH80} the result of Lemma \ref{lem:2.11} is proved via the conservation of the energy, and the argument requires that $\lambda>0$. We note that our proof of Lemma \ref{lem:2.11} is independent of the conservation of the energy and the sign of $\lambda$.
\end{remark}

\subsection{Construction of solutions in $W_1$}
%in the energy space
\label{sec:2.4}
%%%%%%%
We now assume that $\varphi\in W_1{\,\subset H^1(\R^d)}$. From the dominated convergence theorem we have
\begin{align*}
E_\eps(\varphi)\to E(\varphi)\quad\text{as}~\eps\downarrow0,
\end{align*}
where we recall that $E_\eps(\varphi)$ and $E(\varphi)$ are defined in \eqref{eq:2.2} and \eqref{eq:2.5}, respectively.
The nontrivial task is to check the limit of $E_\eps(u_\eps(t))$ as $\eps\downarrow0$. 
We take a function $\theta\in C^{1}_c(\C,\R)$ satisfying
\begin{align}
\label{eq:2.18}
\theta (z)=
\left\{
\begin{aligned}
&1&&\text{if}~|z|\le 1/4,
\\
&0 &&\text{if}~|z|\ge 1/2,
\end{aligned}
\r.
\qquad 0\le\theta (z)\le 1\quad\text{for}~z\in\C,
\end{align}
and set
\begin{align*}
F_{1\eps}(u)&=\theta(u)|u|^2\log\left((|u|+\eps)^2\r),~
F_{2\eps}(u)=(1-\theta(u))|u|^2\log\left((|u|+\eps)^2\r)
\quad\text{for}~\eps>0,
\\
F_1(u)&=\theta(u)|u|^2\log (|u|^2),~
F_2(u)=(1-\theta(u))|u|^2\log (|u|^2).
\end{align*}
%Then $E_\eps(u)$ 
In the following, we restrict the range of $\eps$ to $(0,1/2)$.
The energy of \eqref{eq:2.1} is rewritten as
\begin{align}
\label{eq:2.19}
E_\eps(u)
=\frac{1}{2}\int|\nabla u|^2-\frac{\lambda}{2}\int F_{1\eps}(u)
-\frac{\lambda}{2}\int F_{2\eps}(u)+\frac{\lambda}{2}\int \mu_\eps(|u|).
\end{align}
Taking $\delta>0$ such that $2+\delta\in(2,2^*)$, we obtain
\begin{align}
\label{eq:2.20}
\int|F_2(u)|\cleq \int|u|^{2+\delta}\cleq \norm[u]_{H^1}^{2+\delta}.
%\quad\text{for}~2+\delta\in(2,2^*)
\end{align}
Therefore we note that
\begin{align*}
%%label{eq:2.20}%
\text{if $u\in H^1(\R^d)$, then}~u\in W_1{\iff} \int |F_1(u)|<\infty.
%F_1(u)\in L^1(\R^d)
\end{align*}

We first check the third and fourth terms on the right-hand side of \eqref{eq:2.19}. 
\begin{lemma}
\label{lem:2.13}
For all $t\in\R$ we have
\begin{align*}
\int \mu_\eps(\abs[u_\eps(t)]) \to \int |u(t)|^2,\quad 
\int F_{2\eps}(u_\eps(t)) \to \int F_2(u(t))
\end{align*}
as $\eps\downarrow0$.
\end{lemma}
%%%%%%%%
\begin{proof}
The first convergence is easily proved, for example, by applying the argument in \cite[Lemma 3.3.7]{C03}. From the fact
\begin{align*}
\supp F_{2\eps} \subset \{ z\in\C : |z|\ge1/2\},
\end{align*}
we deduce that for any $\delta\in(0,1)$ there exists $C(\delta)>0$ such that
\begin{align*}
\abs[F_{2\eps}(z)-F_2(w)]\le C(\delta)(|z|^{1+\delta}+|w|^{1+\delta})|z-w|
\quad\text{for all}~z,w\in\C.
\end{align*}
By combining \eqref{eq:2.7} and Lemma \ref{lem:2.10}, this yields the second convergence of the claim.
\end{proof}
%%%%%%%%%
%%%%%%%%%
%The rest is the convergence of the second term in the right-hand side of \eqref{eq:2.19}. Regarding this, the situation is different depending on the sign of $\lambda$.
The following result is already proved in the previous works but we give a self-contained proof to clarify the difference from the case $\lambda<0$.
%%%%%%
\begin{proposition}[\cite{CH80, C83}]
\label{prop:2.14}
Let $\lambda>0$. Then, $u\in (C\cap L^\infty)(\R,W_1)$ and $E(u(t))=E(\varphi)$ for all $t\in\R$.
\end{proposition}
%%%%%%
\begin{proof}
Since $\eps\in(0,1/2)$, we note that ${-F_{1\eps}(u)}$ is the nonnegative term. Hence the first two terms of $E_\eps(u)$ are rewritten as
\begin{align*}
\frac{1}{2}\int|\nabla u|^2-\frac{\lambda}{2}\int F_{1\eps}(u)
=\frac{1}{2}\int|\nabla u|^2+\frac{\lambda}{2}\int |F_{1\eps}(u)|.
\end{align*}
Using the weak lower semicontinuity of the norm for the first term, Fatou's lemma for the second term, and Lemma \ref{lem:2.13}, we deduce that
\begin{align*}
\begin{aligned}
\frac{1}{2}\int|\nabla u(t)|^2+\frac{\lambda}{2}\int |F_1(u(t))|
&\le\liminf_{\eps\to0}\left(E_\eps(u_\eps(t))+\frac{\lambda}{2}\int F_{2\eps}(u_\eps(t))-\frac{\lambda}{2}\int \mu_\eps(u_\eps(t)) \r)
\\
&\le E(\varphi)+\frac{\lambda}{2}\int F_2(u(t))-\frac{\lambda}{2}\int |u(t)|^2
\end{aligned}
\end{align*}
for all $t\in\R$. This yields that
\begin{align*}
u(t)\in W_1,\quad E (u(t))\le E(\varphi)
\quad\text{for all}~t\in\R.
\end{align*}
Combining Lemma \ref{lem:2.10}, we obtain the conservation of the energy
\begin{align}
\label{eq:2.21}
E(u(t))= E(\varphi)\quad\text{for all}~t\in\R.
\end{align}
%%%%
From \eqref{eq:2.20} and \eqref{eq:2.16} we obtain
\begin{align*}
\frac{1}{4}\int|\nabla u(t)|^2+\frac{\lambda}{2}\int |F_1(u(t))|
&\le E(\varphi)+C(\norm[\varphi]_{L^2})
\end{align*}
for all $t\in\R$. Therefore we deduce that
\begin{align*}
u\in L^\infty(\R, H^1(\R^d))\quad\text{and}\quad t\mapsto 
\int |u(t)|^2\log(|u(t)|^2) \in L^\infty(\R),
\end{align*}
which is equivalent to the fact that $u\in L^\infty(\R,W_1)$. Moreover, it follows from \eqref{eq:2.21} and Lemma \ref{lem:2.11} that
\begin{align*}
t\mapsto \int |u(t)|^2\log(|u(t)|^2) \in C(\R),
\end{align*}
which yields $u\in C(\R,W_1)$.
\end{proof}
%%%%%%%
%%%%%%%
For the case $\lambda<0$ the difficulty comes from conflicting signs of the first two terms in $E_\eps(u_\eps)$
\begin{align*}
\frac{1}{2}\int|\nabla u_\eps|^2-\frac{|\lambda|}{2}\int |F_{1\eps}(u_\eps)|
\end{align*}
and the unknown convergence of them. To overcome that, we take advantage of the time-local boundedness of $\{u_\eps\}_{0<\eps<1}$ in $H^1(\R^d)$. 
\begin{proposition}
\label{prop:2.15}
Let $\lambda<0$. Then, $u\in C(\R,W_1)$.
\end{proposition}
%%%%%
\begin{proof}
%{\bf Step 1: $u\in L^\infty_{\rm loc}(\R,W_1)$}. 
First we show that $u\in L^\infty_{\rm loc}(\R,W_1)$. 
It follows from \eqref{eq:2.19} and \eqref{eq:2.4} that for any $T>0$ and $t\in[-T,T]$,
\begin{align*}
\frac{|\lambda|}{2}\int |F_{1\eps}(u_\eps(t))|
&=\frac{\lambda}{2}\int F_{1\eps}(u_\eps(t))
\\
&=-E_\eps(u_\eps(t))+\frac{1}{2}\int|\nabla u_\eps(t)|^2
-\frac{\lambda}{2}\left( \int F_{2\eps}(u_\eps(t))-\int \mu_\eps(|u_\eps(t)|)\r).
\end{align*}
%右辺の第二項以外は$\eps\to0$で極限をもつが，第二項は極限が存在するかはこの段階では分からない．
%%%%
Noting that \eqref{eq:2.7}, we obtain from Fatou's lemma and \eqref{eq:2.7} that
\begin{align*}
\frac{|\lambda|}{2}\int |F_1(u(t))|
\le
\liminf_{\eps\to0}\frac{|\lambda|}{2}\int |F_{1\eps}(u_\eps(t))|
\le -E(\varphi)+C(M_T) 
%\quad\forall t\in [-T,T]
\end{align*}
for all $t\in [-T,T]$. Therefore we deduce that
\begin{align*}
%u\in L^\infty_{\rm loc}(\R,H^1(\R^d))\quad\text{and}\quad
t\mapsto \int |u(t)|^2\log(|u(t)|^2) \in L^\infty_{\rm loc}(\R),
\end{align*}
which shows the desired claim. 

Next we show $u\in C(\R, W_1)$. It is easily verified that
$
t\mapsto \int F_2(u(t)) \in C(\R)$,
and therefore we need to check that $t\mapsto \int F_1(u(t))\in C(\R)$. As in the proof of Lemma \ref{lem:2.11}, it is enough to show the continuity at $t=0$. From the calculation in the previous paragraph we obtain
\begin{align*}
\frac{|\lambda|}{2}\int |F_{1\eps}(u_\eps(t))|
&=-E_\eps(u_\eps(t))+\frac{1}{2}\int|\nabla u_\eps(t)|^2
-\frac{\lambda}{2}\left( \int F_{2\eps}(u_\eps(t))-\int \mu_\eps(|u_\eps(t)|)\r)
\\
&\le -E_\eps(\varphi)+\frac{1}{2}e^{4\abs[\lambda]\abs[t]}\norm[\nabla\varphi]_{L^2}^2
-\frac{\lambda}{2}\left( \int F_{2\eps}(u_\eps(t))-\int \mu_\eps(|u_\eps(t)|)\r)
\end{align*}
for all $t\in\R$. It follows from Fatou's lemma and Lemma \ref{lem:2.13} that
\begin{align*}
\frac{|\lambda|}{2}\int |F_1(u(t))|
&\le -E(\varphi)+\frac{1}{2}e^{4\abs[\lambda]\abs[t]}\norm[\nabla\varphi]_{L^2}^2
-\frac{\lambda}{2}\int F_{2}(u(t))+\frac{\lambda}{2}\int |u(t)|^2.
\end{align*}
Passing to the limit as $t\to0$ we deduce that
\begin{align*}
\limsup_{t\to0} \frac{|\lambda|}{2}\int |F_1(u(t))|
&\le -E(\varphi)+\frac{1}{2}\norm[\nabla\varphi]_{L^2}^2
-\frac{\lambda}{2}\int F_{2}(\varphi)+\frac{\lambda}{2}\int |\varphi|^2
\\
&=  \frac{\lambda}{2}\int F_1(\varphi)=\frac{|\lambda|}{2}\int \abs[F_1(\varphi)].
\end{align*}
On the other hand, it follows from Fatou's lemma that
$
t\mapsto\int |F_1(u(t))|
%\frac{|\lambda|}{2}\int \abs[F_1(\varphi)]\le\liminf_{t\to0} \frac{|\lambda|}{2}\int |F_1(u(t))|.
$
is lower semicontinuous. Therefore, combining with the previous calculation, we deduce that $t\mapsto \int F_1(u(t))$ is continuous at $t=0$. Hence, it follows from Lemma \ref{lem:2.11} that $u\in C(\R,W_1)$.
%Combined with Lemma \ref{lem:2.11}, this yields that $u\in C(\R,W_1)$.
\end{proof}
Finally, it follows from $u\in C(\R, W_1)$ and \cite[Lemma 2.6]{C83} that 
\begin{align*}
i\pt_t u+\Delta u+\lambda u\log(|u|^2)=0\quad\text{in}~W_1^*
\end{align*}
for all $t\in\R$. This completes the proof of Theorem \ref{thm:1.1}. 
%%%%%%%%%%
%%%%%%%%%%
\begin{remark}
\label{rem:2.16}
Unlike the case of $\lambda>0$, $u\in L^\infty(\R,W_1)$ cannot hold when $\lambda<0$. In fact it is proved in \cite{CG18} that 
%for $\lambda<0$ and $\varphi\in H^1(\R^d)\cap\scF(H^1)\subset W_1$, 
in this case the solution of \eqref{eq:1.1} for $\varphi\in H^1(\R^d)\cap\scF(H^1)$ satisfies 
\begin{align*}
\norm[\nabla u(t)]_{L^2}^2 \sim \log t 
\quad\text{as}~t\to\infty.
\end{align*}
\end{remark}
%%%%%%%%%%%%%
%%%%%%%%%%%%%
\begin{remark}
\label{rem:2.17}
When $\lambda<0$, it seems non-trivial whether the conservation of the energy holds in the solution class of $W_1$.
% and not easy to justify it. 
If we consider 
$\varphi\in H^1(\R^d)\cap\scF(H^\alpha)$ for some $\alpha\in(0,1]$, then one can prove that
\begin{align*}
%\int |F_1(u (t))|<\infty,\quad 
\int F_{1\eps}(u_\eps (t))\to \int F_1(u (t))
\quad\text{for}~ t\in\R,
\end{align*}
which yields \eqref{eq:2.21}. However, this discussion is done in a strictly narrower solution class than $W_1$.
\end{remark}

%%%%%%%%%%%%%%%%%%%%
%%%%%%%%%%%%%%%%%%%%

\section{The Cauchy problem in the $H^2$-energy space}
%Construction of solutions in the $H^2$-energy space
\label{sec:3}

In this section we construct the solution of \eqref{eq:1.1} in the $H^2$-energy space $W_2$.
We first note that for any $\eps>0$ and $\varphi\in W_2\subset H^2(\R^d)$ there exists a unique solution 
\begin{align*}
u_\eps\in C(\R, H^2(\R^d))\cap C^1(\R, L^2(\R^d))
\end{align*}
to \eqref{eq:2.1}. This is proved by applying the regularity result\footnote{For example one can apply the argument of \cite[Proposition 4.3.9]{CH98} to \eqref{eq:2.1}.} to the $H^1$-solution constructed in Section \ref{sec:2.1}.
%We note that $u\mapsto g_\eps(u)$は$H^1$有界集合上で，$L^2$から$L^2$へのリプシッツ連続となることに注意する．
%%
Alternatively, the $H^2$-solution of \eqref{eq:2.1} can be reconstructed in the manner described in Section \ref{sec:2.1}.
%%%%%%
%%%%%%
\subsection{Construction of weak $H^2$-solutions}
%Energy estimates of time derivatives
\label{sec:3.1}
As seen in \cite{K87, C03} we obtain $H^2$-estimates by $L^2$-estimates of time derivatives. We first derive analogous estimates of Lemma \ref{lem:2.4} for time derivatives.
\begin{lemma}[\cite{CG18}]
\label{lem:3.1}%
For all $t\in\R$ we have
\begin{align}
\label{eq:3.1}%
\norm[\pt_t u_\eps(t)]_{L^2}^2 
\le e^{4\abs[\lambda]\abs[t]}\norm[\pt_tu_\eps (0)]_{L^2}^2.
%\quad\forall t\in\R.
\end{align}
\end{lemma}
%%%%%%%%%%
%%%%%%%%%%
\begin{proof}
For completeness we give a proof. Similarly to the proof of Lemma~\ref{lem:2.4}, we have
\begin{align*}
\frac{d}{dt}\norm[\pt_t u_\eps]_{L^2}^2
&=2\Im\rbra[i\pt_t^2u_\eps,\pt_t u_\eps]
\\
&=2\Im\rbra[\pt_t\left\{-\Delta u_\eps -2\lambda u_\eps\log(|u_\eps|+\eps) \r\},\pt_t u_\eps]
\\
&=-4\lambda\Im\rbra[\frac{u_\eps}{|u_\eps|+\eps}\pt_t|u_\eps|,\pt_t u_\eps].
\end{align*}
Noting that $\abs[\pt_t |u_\eps|]\le |\pt_t u_\eps|$ we obtain
\begin{align*}
\frac{d}{dt}\norm[\pt_t u_\eps(t)]_{L^2}^2\le 4|\lambda|\norm[\pt_t u_\eps(t)]_{L^2}^2\quad\text{for all}~t\in\R.
\end{align*}
Then the result follows from Gronwall's lemma.
\end{proof}
%%%%%%%%%
%%%%%%%%%
It follows from \eqref{eq:3.1} and the equation \eqref{eq:2.1} that
\begin{align*}
\norm[\pt_t u_\eps(t)]_{L^2}^2 
&\le e^{4\abs[\lambda]\abs[t]}
\norm[-\Delta\varphi-2\lambda\varphi\log(|\varphi|+\eps)]_{L^2}^2 
\\
&\cleq e^{4\abs[\lambda]\abs[t]} \left(\norm[\Delta\varphi]_{L^2}^2+\norm[\varphi \log(|\varphi|^2) ]_{L^2}^2 \r).
\end{align*}
Combined with \eqref{eq:2.7}, this yields that 
%\eqref{eq:2.3} and \eqref{eq:2.6}
\begin{align}
\label{eq:3.2}%
N_T\ce&\sup_{\eps\in(0,1)}
\left(\norm[u_\eps]_{C_T(H^1)}+\norm[\pt_t u_\eps]_{C_T(L^2)} \r)
\le C(T, \norm[\varphi]_{H^2},\norm[\varphi \log(|\varphi|^2) ]_{L^2})
\end{align}
for any $T>0$.
Applying the argument in Sections~\ref{sec:2.2} and \ref{sec:2.3}, there exists a unique solution $u\in C(\R,H^1(\R^d))$ to \eqref{eq:1.1} in the sense of \eqref{eq:1.9}, and $u_\eps(t)\wto u(t)$ weakly in $H^1(\R^d)$ for all $t\in\R$.
%%%
Moreover, we deduce from \eqref{eq:3.2} that 
\begin{gather*}
u\in W^{1,\infty}_{\rm loc}(\R,L^2(\R^d)),
\quad
\pt_t u_\eps(t) \wto \pt_t u(t)\quad\text{in}~L^2(\R^d)
\end{gather*}
for all $t\in\R$. Furthermore it follows from \eqref{eq:1.9} that  
\begin{align*}
\pt_t u \in C_w(\R, L^2_{\rm loc}(\R^d)),\quad \Delta u \in (C_w\cap L^\infty_{\rm loc})(\R, L^2_{\rm loc}(\R^d))
\end{align*}
and \eqref{eq:1.10} holds for any $\Omega\subset\subset\R^d$.
%%%%
%This completes the proof of the first part of Theorem~\ref{thm:1.2}.
%%%%%%%%%%%
\begin{remark}
\label{rem:ch}
In \cite{CH80} (see also \cite{Haraux81}) the authors apply the theory of maximal monotone operators to construct the solution belonging to the set
\begin{align*}
Z=\left\{\varphi\in L^2(\R^d)\cap H^1_{\rm loc}(\R^d) : \Delta \varphi+\lambda\varphi\log(|\varphi|^2) \in L^2(\R^d) \r\}.
\end{align*}
One can recover this result by using the equation \eqref{eq:2.1} and the limiting procedure based on \eqref{eq:3.1}.
\end{remark}

%%%%%%%%%%%
%%%%%%%%%%%
\subsection{A priori estimates in $W_2$}
\label{sec:3.2}
We now construct solutions in the $H^2$-energy space $W_2$ when $\lambda>0$.
The following lemma plays a key role in obtaining a priori estimates in $W_2$. 
\begin{lemma}
\label{lem:3.2}
For all $t\in\R$ the following relation holds:
\begin{align}
\label{eq:3.3}%
\norm[\pt_t u_\eps]_{L^2}^2
=
\begin{aligned}[t]
&\norm[\Delta u_\eps]_{L^2}^2-4\lambda\Re\rbra[\nabla u_\eps\cdot\nabla|u_\eps|, \frac{u_\eps}{|u_\eps|+\eps}]_{L^2}
\\
&{}~-4\lambda\rbra[ |\nabla u_\eps|^2,\log(|u_{\eps}|+\eps)]_{L^2}+4\lambda^2\norm[u_\eps\log(|u_\eps|+\eps)]_{L^2}^2,
\end{aligned}
\end{align}
where $u_\eps$ is abbreviated as $u_\eps=u_\eps(t)$.
%where we wrote $u_\eps=u_\eps(t)$
\end{lemma}
%%%%%
\begin{proof}
%ここでは$u=u_\eps$と書いて計算する．
It follows from \eqref{eq:2.1} that
\begin{align*}
\norm[\pt_t u_\eps]_{L^2}^2&=\norm[\Delta u_\eps+2\lambda u_\eps\log(|u_\eps|+\eps)]_{L^2}^2
\\
&=\norm[\Delta u_\eps]_{L^2}^2+4\lambda\Re\rbra[\Delta u_\eps,u_\eps\log(|u_\eps|+\eps)]_{L^2}+4\lambda^2\norm[u_\eps\log(|u_\eps|+\eps)]_{L^2}^2.
\end{align*}
%We rewrite the second term in the right-hand side as
We note that
\begin{align*}
2\Re\rbra[\Delta u_\eps,u_\eps\log(|u_\eps|+\eps)]_{L^2}&=\rbra[\Delta|u_\eps|^2-2\abs[\nabla u_\eps]^2, \log(|u_\eps|+\eps)]_{L^2},
\\
\rbra[\Delta|u_\eps|^2, \log(|u_\eps|+\eps)]_{L^2}
&=-\rbra[\nabla(|u_\eps|^2),\frac{\nabla|u_\eps|}{|u_\eps|+\eps}]_{L^2}
%\\
=-2\Re\rbra[\overline{u_\eps}\nabla u_\eps,\frac{\nabla|u_\eps|}{|u_\eps|+\eps}]_{L^2}.
\end{align*}
Combining these equalities, we obtain that
%Therefore we deduce that
\begin{align*}
\norm[\pt_t u_\eps]_{L^2}^2&=\norm[\Delta u_\eps]_{L^2}^2
+4\lambda\Re\rbra[\Delta u_\eps,u_\eps\log(|u_\eps|+\eps)]_{L^2}+4\lambda^2\norm[u_\eps\log(|u_\eps|+\eps)]_{L^2}^2
\\
&=
\begin{aligned}[t]
&\norm[\Delta u_\eps]_{L^2}^2
-4\lambda\Re\rbra[\overline{u_\eps}\nabla u_\eps,\frac{\nabla|u_\eps|}{|u_\eps|+\eps}]_{L^2}
\\
&{}~-4\lambda\rbra[|\nabla u_\eps|^2,\log(|u_\eps|+\eps)]_{L^2}
+4\lambda^2\norm[u_\eps\log(|u_\eps|+\eps)]_{L^2}^2.
\end{aligned}
\end{align*}
This completes the proof.
\end{proof}
%%%%
%We now assume that $\lambda>0$ and 
Based on the relation \eqref{eq:3.3}, we prove that $u(t)\in W_2$ for all $t\in\R$ through the limit $\eps\downarrow0$. 
\begin{proposition}
\label{prop:3.3}
Let $\lambda>0$. Then $u\in L^\infty_{\rm loc}(\R, W_2)$.
\end{proposition}
%%%%%
\begin{proof}
We take the same function $\theta\in C^1(\C, \R)$ defined in \eqref{eq:2.18}. 
For $\eps\in(0,1/2)$ the third term on the right-hand side 
of \eqref{eq:3.3} is estimated as
\begin{align*}
-4\lambda\rbra[ |\nabla u_\eps|^2,\log(|u_{\eps}|+\eps)]_{L^2}
&=
-4\lambda\int\bigl(\theta(u_\eps)+(1-\theta(u_\eps)) \bigr)|\nabla u_\eps|^2\log(|u_{\eps}|+\eps)
\\
&\ge -4\lambda\int (1-\theta(u_\eps) )|\nabla u_\eps|^2\log(|u_{\eps}|+\eps),
\end{align*}
where we dropped the integral around the origin. 
%because it is nonnegative. 
For $\delta>0$ determined later we have
\begin{align*}
\left|-4\lambda\int (1-\theta(u_\eps) )|\nabla u_\eps|^2\log(|u_{\eps}|+\eps)^2\r|
&\cleq  \int_{|u_\eps|\ge1/4} \abs[\nabla u_\eps]^2\log(1+|u_\eps|)^2
\\
&\cleq\int_{|u_\eps|\ge1/4}\abs[\nabla u_\eps]^2|u_\eps|^\delta.
\end{align*}
To estimate the last term, we use H\"older's inequality and the interpolation inequality
%Gagliardo--Nirenberg type inequality
\begin{align}
\label{eq:3.4}
\norm[\nabla f]_{L^{2p'}}\cleq \norm[\Delta f]_{L^2}^{\frac{1}{2}+\frac{d}{4p}} \norm[f]_{L^2}^{\frac{1}{2}-\frac{d}{4p}},
\end{align}
where $d/(4p)\in (0,1/2)$ and $1/p'=1-1/p$.
Then, we obtain
\begin{align*}
\left|\int\abs[\nabla u_\eps]^2|u_\eps|^\delta\r|
\le \norm[\nabla u_\eps]_{L^{2p'}}^2\norm[u_\eps]_{L^{p\delta}}^\delta
\cleq \norm[\Delta u_\eps]_{L^2}^{ 1+\frac{d}{2p}}
\norm[u_\eps]_{L^2}^{1-\frac{d}{2p}}
\norm[u_\eps]_{L^{p\delta}}^\delta.
\end{align*}
We now fix $p$ satisfying $1+d/(2p)<2$ and set $\delta=2/p$. Combined with \eqref{eq:2.3}, this yields that
\begin{align*}
\left|\int\abs[\nabla u_\eps]^2|u_\eps|^\delta\r|\cleq 
\norm[\Delta u_\eps]_{L^2}^{1+\frac{d}{2p}}\norm[\varphi]_{L^2}^{1-\frac{d}{2p}+\frac{2}{p}}
\le \frac{1}{2}\norm[\Delta u_\eps]^2+C(\norm[\varphi]_{L^2}).
\end{align*}
Gathering these estimates, we obtain
\begin{align*}
-4\lambda\rbra[ |\nabla u_\eps|^2,\log(|u_{\eps}|+\eps)]_{L^2}
\geq -\frac{1}{2}\norm[\Delta u_\eps]^2-C(\norm[\varphi]_{L^2}).
\end{align*}
%%%%

Now using \eqref{eq:3.3} we deduce that 
\begin{align*}
\norm[\pt_t u_\eps]_{L^2}^2+4\lambda\norm[\nabla u_\eps]_{L^2}^2+C(\norm[\varphi]_{L^2})
\ge \frac{1}{2}\norm[\Delta u_\eps]_{L^2}^2+4\lambda^2\norm[u_\eps\log(|u_\eps|+\eps)]_{L^2}^2.
\end{align*}
Therefore, we deduce from \eqref{eq:3.2} that
\begin{align}
\label{eq:3.5}
\norm[\Delta u_\eps(t)]_{L^2}^2+4\lambda^2\norm[u_\eps(t)\log(|u_\eps(t)|+\eps)]_{L^2}^2\le C(N_T)<\infty
%\quad \forall t\in[-T,T].
\end{align}
for all $t\in[-T,T]$.
%%%
Similarly to Lemma \ref{lem:2.5} we deduce from \eqref{eq:3.5} that
\begin{align}
\label{eq:3.6}
\Delta u\in L^\infty_{\rm loc}(\R, L^2(\R^d)),\quad \Delta u_\eps(t)\wto \Delta u(t)\quad\text{in}~L^2(\R^d)~\text{as}~\eps\downarrow0
%~\forall t\in\R.
\end{align}
for all $t\in\R$. Therefore it follows from \eqref{eq:3.5}, the weak lower semicontinuity of the norm, and Fatou's lemma that
\begin{align*}
\norm[\Delta u(t)]_{L^2}^2+\lambda^2\norm[u(t)\log (|u(t)|^2) ]_{L^2}^2\le C(N_T)\
%quad \forall t\in[-T,T].
\end{align*}
for $t\in[-T,T]$. Hence $u\in L^\infty_{\rm loc}(\R,W_2)$.
%we deduce that $u(t)\in W_2$ for all $t\in\R$.
\end{proof}
%%%%%%%%%%%%%%%
%%%%%%%%%%%%%%%
\begin{remark}
\label{rem:3.4}%
The second term on the right-hand side of \eqref{eq:3.3} with $\eps=0$ is formally rewritten as
\begin{align*}
\Re\rbra[\nabla u,\nabla|u|\frac{u}{|u|}]_{L^2}=
\Re\rbra[\frac{\overline{u}\nabla u}{|u|},\nabla|u|]_{L^2}=\norm[\nabla|u|]_{L^2}^2.
\end{align*}
Therefore we can see that \eqref{eq:1.11} is formally obtained from \eqref{eq:3.3} as $\eps\downarrow0$. This observation is helpful in the discussion in the next subsection.
\end{remark}

\subsection{Further regularity}
\label{sec:3.3}
 Our goal here is to prove that $u\in C(\R,W_2)$. Throughout this subsection we assume that $\lambda>0$. It follows from Proposition \ref{prop:3.3} that $u$ satisfies 
% the equation \eqref{eq:1.1} in $L^2(\R^d)$: 
\begin{align}
\label{eq:3.7}
i\pt_t u+\Delta u+\lambda u\log(|u|^2)=0\quad\text{in}~L^2(\R^d)
\end{align}
for a.e. $t\in\R$. It also follows from \eqref{eq:3.2} and \eqref{eq:3.7} that
\begin{align*}
\sup_{\eps\in(0,1)}\norm[u_\eps]_{L^\infty_T(H^2)}
\le C(T, \norm[\varphi]_{H^2},\norm[\varphi \log(|\varphi|^2) ]_{L^2}). 
\end{align*}
Combined with Lemma \ref{lem:2.10}, this yields that
\begin{align}
\label{eq:3.8}
u_\eps (t) \to u(t)\quad \text{in}~H^1(\R^d) \quad\text{for all}~t\in\R.
\end{align}
\begin{lemma}
\label{lem:3.5}
$u\in C_w(\R, H^2(\R^d))$.
%$\pt_tu\in C_w(\R,L^2(\R^d))$
\end{lemma}
\begin{proof}
The claim follows from
$u\in C(\R, H^1(\R^d))$, $u\in L^\infty_{\rm loc}(\R,W_2)$, and a density argument. 
%It follows from $u(t)\log|u(t)|^2\in L^2(\R^d)$ for $t\in\R$ and 
\end{proof}
Similarly to the proofs of Lemma \ref{lem:2.11} and Proposition \ref{prop:2.15}, it is enough to show the strong continuity at $t=0$. Since we know the weak continuity, the key in the proof is to show that
\begin{align}
\label{eq:3.9}
\limsup_{t\to0}\norm[\Delta u(t)]_{L^2}^2\le\norm[\Delta\varphi]_{L^2}^2.
\end{align}
To prove this claim, we take advantage of \eqref{eq:3.3}.
%%%
In order to proceed with the proof, we need to ensure the convergence of the second term on the right-hand side of \eqref{eq:3.3}. We recall that if $f\in H^1(\R^d)$, then $|f|\in H^1(\R^d)$ and
%$|f|$ belongs to $H^1(\R^d)$
\begin{align}
\label{eq:3.10}
\nabla |f|(x)=
\left\{
\begin{aligned}
&\frac{\Re(\overline{f(x)}\nabla f(x))}{|f|(x)}&&\text{if}~f(x)\neq0,
\\
&0&&\text{if}~f(x)=0,
\end{aligned}
\r.
\end{align}
see, e.g., \cite[Theorem 6.17]{LL}.
\begin{lemma}
\label{lem:3.6}
For all $t\in\R$ we have
\begin{align*}
\Re\rbra[\nabla u_\eps(t)\cdot\nabla|u_\eps|(t), \frac{u_\eps(t)}{|u_\eps(t)|+\eps}]_{L^2} \to \norm[\nabla|u|(t)]_{L^2}^2
\quad\text{as}~\eps\downarrow0.
\end{align*}
%as $\eps\downarrow0$.
\end{lemma}
%%%%
\begin{proof}
We fix $t\in\R$ and set $f_\eps= u_\eps(t)$ and $f= u(t)$. Since by \eqref{eq:3.8} 
\begin{align}
\label{eq:3.11}
\nabla f_\eps\to\nabla f \quad\text{in}~L^2(\R^d),
\end{align}
the key in the proof is to show that 
\begin{align}
\label{eq:3.12}
\nabla|f_\eps|\frac{\overline{f_\eps} }{|f_\eps|+\eps}
\wto \nabla|f|\sgn f
%\frac{f}{|f|}
\quad \text{in}~L^2(\R^d),
\end{align}
where $\sgn$ is the signum function
\begin{align*}
\sgn f\ce
\left\{
\begin{aligned}
&\overline{f}/|f| &&\text{if}~f\neq0,
\\
&0 &&\text{if}~f=0.
\end{aligned}
\r.
\end{align*}
It follows from \eqref{eq:3.8} that there exists a sequence $\eps_n\downarrow0$
%$\{\eps_n\}$ with $\eps_n\downarrow0$ 
such that 
\begin{align*}
f_{\eps_n}\to f, \quad \nabla f_{\eps_n}\to \nabla f\quad \text{a.e. on}~\R^d.
\end{align*}
We note that
\begin{align*}
\abs[\nabla|f_\eps|\frac{\overline{f_\eps} }{|f_\eps|+\eps}]\le \abs[\nabla f_\eps]
%\quad\text{a.e.}~x\in\R^d.
\end{align*}
and that
\begin{align*}
\nabla|f_{\eps_n}|\frac{\overline{f_{\eps_n}} }{|f_{\eps_n}|+{\eps_n}}
&\to\nabla|f|\sgn f \quad\text{a.e. on}~\R^d,
\\
\abs[\nabla f_{\eps_n}] &\to \abs[\nabla f]\quad\text{in}~L^2(\R^d).
\end{align*}
Thus, applying a generalized dominated convergence theorem (see \cite[p.59]{Fo99}), we deduce that for any $\psi\in L^2(\R^d)$
\begin{align*}
\int \nabla|f_{\eps_n}|\frac{\overline{f_{\eps_n}} }{|f_{\eps_n}|+{\eps_n}}\psi \to \int \nabla|f|(\sgn f) \psi .
%\quad\text{as}~n\to\infty.
\end{align*}
It is easily verified that this limit does not depend on the sequence  $\eps_n\downarrow0$, 
%the sequence $\{\eps_n\}$ with $\eps_n\downarrow0$
so we deduce that \eqref{eq:3.12} holds as $\eps\downarrow0$. Hence, we obtain from \eqref{eq:3.11} and \eqref{eq:3.12} that
\begin{align*}
\Re\rbra[\nabla f_\eps\cdot\nabla|f_\eps|, \frac{f_\eps}{|f_\eps|+\eps}]_{L^2}\to\Re\rbra[\nabla f\cdot\nabla|f|,\sgn f]_{L^2}
=\norm[\nabla|f|]_{L^2}^2,
\end{align*}
where we used \eqref{eq:3.10} in the last equality.
\end{proof}
%%%%%%%%%%%%
%%%%%%%%%%%%
From \eqref{eq:3.3} we obtain the following through the limit $\eps\downarrow0$ .
\begin{lemma}
\label{lem:3.7}
For all $t\in\R$ the following relation holds:
\begin{align}
\label{eq:3.13}%
e^{\lambda\abs[t]}\norm[\Delta\varphi+\lambda\varphi\log(|\varphi|^2)]_{L^2}^2
\ge
\begin{aligned}[t]
&\norm[\Delta u]_{L^2}^2-4\lambda\norm[\nabla|u|]_{L^2}^2
\\
&{}-2\lambda\rbra[ |\nabla u|^2,\log(|u|^2)]_{L^2}+\lambda^2\norm[u\log(|u|^2)]_{L^2}^2,
\end{aligned}
\end{align}
where $u$ is abbreviated as $u=u(t)$.
%where $u$ on the right-hand side is abbreviated as $u=u(t)$.
\end{lemma}
%%%%%%%%%%%%%
\begin{proof}
We note from Lemma \ref{lem:3.1} and \eqref{eq:2.1} that
\begin{align*}
\norm[\pt_t u_\eps]_{L^2}^2 
\le e^{4\lambda\abs[t]}\norm[\Delta\varphi+2\lambda\varphi\log(|\varphi|+\eps)]_{L^2}^2.
\end{align*}
By \eqref{eq:3.6} the first term on the right-hand side in \eqref{eq:3.3} is estimated as
\begin{align*}
\norm[\Delta u]_{L^2}^2\le\liminf_{\eps\downarrow0}\norm[\Delta u_\eps]_{L^2}^2,
\end{align*}
and by Fatou's lemma the fourth term is estimated as 
\begin{align*}
4\lambda^2\norm[u\log |u|]_{L^2}^2\le\liminf_{\eps\downarrow0}4\lambda^2\norm[u_\eps\log(|u_\eps|+\eps)]_{L^2}^2.
\end{align*}
As in the proof of Proposition \ref{prop:3.3}, we decompose the third term on the right-hand side in \eqref{eq:3.3} as
\begin{align*}
-4\lambda\rbra[ |\nabla u_\eps|^2,\log(|u_{\eps}|+\eps)]_{L^2}
&=
-4\lambda\int\bigl(\theta(u_\eps)+(1-\theta(u_\eps)) \bigr)|\nabla u_\eps|^2\log(|u_{\eps}|+\eps)
\\
&=I_{1\eps}+I_{2\eps},
\end{align*}
where 
\begin{align*}
I_{1\eps}&\ce -4\lambda\int \theta(u_\eps)|\nabla u_\eps|^2\log(|u_{\eps}|+\eps),
\\
I_{2\eps}&\ce -4\lambda\int (1-\theta(u_\eps))|\nabla u_\eps|^2\log(|u_{\eps}|+\eps).
\end{align*}
Applying Fatou's lemma, we have
\begin{align*}
-4\lambda\int \theta(u)|\nabla u|^2\log|u|
\le \liminf_{\eps\downarrow0} I_{1\eps}.
\end{align*} 
As for $I_{2\eps}$, it is easily verified from \eqref{eq:3.8} that
\begin{align*}
I_{2\eps} \to  -4\lambda\int (1-\theta(u))|\nabla u|^2\log|u|.
\end{align*}
Gathering these estimates, we deduce from \eqref{eq:3.3} and Lemma \ref{lem:3.7} that
\begin{align*}
e^{\lambda\abs[t]}\norm[\Delta\varphi+\lambda\varphi\log|\varphi|^2]_{L^2}^2\ge\liminf_{\eps\downarrow0}\bigl( \text{RHS of \eqref{eq:3.3}} \bigr)
\ge \bigl( \text{RHS of \eqref{eq:3.13}} \bigr),
\end{align*}
which completes the proof.
\end{proof}
Similarly to Lemma \ref{lem:3.6}, one can prove the following.
\begin{lemma}
\label{lem:3.8}
$t\mapsto\norm[\nabla|u|(t)]_{L^2}^2\in C(\R)$.
\end{lemma}
%%%%%%%%%%%%
%%%%%%%%%%%%
\begin{proof}
%The argument is similar to the proof of Lemma \ref{lem:3.6}. 
We fix $t_0\in\R$ and consider any sequence $t_n\to t_0$.
%$\{t_n\}$ with $t_n\to t_0$. 
We set $f_n=u(t_n)$ and $f=u(t_0)$. Since $u\in C(\R,H^1(\R^d))$, there exists a subsequence $t_{n_k}\to t_0$ such that
\begin{align*}
f_{n_k}\to f,\quad \nabla f_{n_k}\to\nabla f\quad\text{a.e. on}~\R^d.
\end{align*}
We obtain from \eqref{eq:3.10} and \eqref{eq:3.8} that
\begin{gather*}
\nabla\abs[f_{n_k}]\to\nabla\abs[f]\quad\text{a.e. on}~\R^d,
\\
\abs[\nabla\abs[f_{n_k}] ]\le \abs[\nabla f_{n_k}],
~\text{and}~\norm[\nabla f_{n_k}]_{L^2}^2\to\norm[\nabla f]_{L^2}^2.
%\quad\text{a.e.}~x\in\R^d.
\end{gather*}
Thus, applying a generalized dominated convergence theorem (see \cite[p.59]{Fo99}), we deduce that
\begin{align*}
\norm[\nabla{\abs[f_{n_k}]} ]_{L^2}^2 \to  \norm[\nabla{\abs[f]} ]_{L^2}^2.
\end{align*}
This limit does not depend on the subsequence, so $\norm[\nabla {\abs[f_n]} ]_{L^2}^2 \to  \norm[\nabla{\abs[f]} ]_{L^2}^2$. This completes the proof. 
\end{proof}
We are now in a position to prove \eqref{eq:3.9}.
%%%%%%
\begin{proof}[Proof of \eqref{eq:3.9}]
We consider the limit $t\to0$ in \eqref{eq:3.13}. Similarly to the proof of Lemma \ref{lem:3.7}, we obtain
\begin{align}
\label{eq:3.14}
\lambda^2\norm[\varphi\log (|\varphi|^2)]_{L^2}^2&\le\liminf_{t\to0}\lambda^2\norm[u\log(|u|^2)]_{L^2},
\\
-2\lambda\rbra[ |\nabla\varphi|^2,\log(|\varphi|^2) ]_{L^2}
&\le\liminf_{t\to0}\bigl\{-2\lambda\rbra[ |\nabla u|^2,\log(|u|^2) ]_{L^2} \bigr\}.
\notag
\end{align}
Combining these with Lemma \ref{lem:3.8}, we deduce that
\begin{align*}
\norm[\Delta\varphi+\lambda\varphi\log(|\varphi|^2)]_{L^2}^2
&\ge\limsup_{t\to0}\bigl(\text{RHS of \eqref{eq:3.13}}\bigr)
\\
&\ge
\begin{aligned}[t]
&\limsup_{t\to0}\norm[\Delta u]_{L^2}^2-4\lambda\norm[\nabla|\varphi|]_{L^2}^2
\\
&{}~-2\lambda\rbra[ |\nabla\varphi|^2,\log(|\varphi|^2)]_{L^2}+\lambda^2\norm[\varphi\log(|\varphi|^2) ]_{L^2}^2.
\end{aligned}
\end{align*}
Similarly to the proof of Lemma \ref{lem:3.2}, the left-hand side is expanded as
%\footnote{The integration by parts is justified by a suitable regularization}
\begin{align*}
\norm[\Delta\varphi+\lambda\varphi\log(|\varphi|^2)]_{L^2}^2
=
\begin{aligned}[t]
&\norm[\Delta\varphi]_{L^2}^2-4\lambda\norm[\nabla|\varphi|]_{L^2}^2
\\
&{}~-2\lambda\rbra[ |\nabla\varphi|^2,\log(|\varphi|^2) ]_{L^2}+\lambda^2\norm[\varphi\log(|\varphi|^2) ]_{L^2}^2,
\end{aligned}
\end{align*}
where the integration by parts in the argument is justified by a suitable regularization. Hence, combining the inequalities above, we conclude the claim.
\end{proof}
%%%%%%%%%%%%
We now complete the proof of Theorem \ref{thm:1.2}. From Lemma \ref{lem:3.5} and \eqref{eq:3.9} we obtain
\begin{align*}
\lim_{t\to0}\norm[\Delta u]_{L^2}^2=\norm[\Delta\varphi]_{L^2}^2,
\end{align*}
which implies the strong continuity in $H^2(\R^d)$ at $t=0$. Hence $u\in C(\R,H^2(\R^d))$. 

In a similar way to the proof of \eqref{eq:3.9} we deduce that
\begin{align*}
\limsup_{t\to0}\norm[u\log(|u|^2)]_{L^2}^2\le\norm[\varphi\log(|\varphi|^2)]_{L^2}^2.
\end{align*}
Combined with \eqref{eq:3.14}, this yields
\begin{align}
\label{eq:3.15}
\lim_{t\to0}\norm[u\log(|u|^2) ]_{L^2}^2=\norm[\varphi\log(|\varphi|^2) ]_{L^2}^2.
\end{align}
Therefore we deduce that $u\in C(\R,W_2)$. Moreover, it follows from the equation \eqref{eq:3.7} that $\pt_t u\in C(\R,L^2(\R^d))$, and then \eqref{eq:3.7} holds for all $t\in\R$. This completes the proof of Theorem~\ref{thm:1.2}.

\begin{remark}
\label{rem:3.9}
Similarly to the proof of Lemma \ref{lem:2.11}, it follows from \eqref{eq:3.1} that
\begin{align*}
\limsup_{t\to0}\norm[\pt_t u(t)]_{L^2}^2\le \norm[\pt_t u(0)]_{L^2}^2.
\end{align*}
Combined with the weak continuity of $t\mapsto\pt_tu(t)\in L^2(\R^d)$, this yields the strong continuity of this function with values in $L^2(\R^d)$. Therefore, we obtain from \eqref{eq:3.7} and \eqref{eq:3.9} that \eqref{eq:3.15}. This gives a slightly different proof of $u\in C(\R,W_2)$.
\end{remark}

%%%%%%%%%%%%%%
%%%%%%%%%%%%%%

\section{The Cauchy problem in a general domain}
%The case of a general domain
\label{sec:4}

Our main results still hold for a general domain without particularly changing the proof and in this section we restate Theorem \ref{thm:1.1} and \ref{thm:1.2} for this case.  

Let $\Omega$ be any domain. 
%As in \cite[Theorem 9.3.4]{C03}, 
We consider the logarithmic Schr\"odinger equation 
\begin{align}
\label{eq:4.1}
\left\{
\begin{aligned}
&i\pt_t u+\Delta u+\lambda u\log (|u|^2)=0,\quad(t,x)\in\R\times\Omega,
%~\lambda\in\R\setminus\{0\},
\\
&u(0,x)=\varphi(x),
%\quad\text{where}~\lambda\in\R\setminus\{0\}.
\end{aligned}
\r.
\end{align}
with zero Dirichlet boundary condition $\bigl.u\bigr|_{\pt\Omega}=0$. 
When we consider \eqref{eq:4.1} in $H^1$ setting, the Laplacian $\Delta$ is understood to be the self-adjoint realization in the Hilbert space $H^{-1}(\Omega)$ with the domain $D(\Delta)=H^1_0(\Omega)$ (see \cite[Chapter 2]{CH98}). The energy space for \eqref{eq:4.1} is defined as in the case of the whole space by
\begin{align*}
W_1=W_1(\Omega)=\left\{ \varphi\in H^1_0(\Omega) :|\varphi|^2\log(|\varphi|^2)\in L^1(\Omega) \r\}.
\end{align*}
Then Theorem \ref{thm:1.1} can be restated for \eqref{eq:4.1} as follows.
\begin{theorem}
\label{thm:4.1}
%\leavevmode\\
%\indent
Let $\lambda\in\R\setminus\{0\}$ and let $\Omega$ be a general domain. For any $\varphi\in H^1_0(\Omega)$, there exists a unique solution $C(\R, H^1_0(\Omega))$
to \eqref{eq:4.1} in the sense of
\begin{align}
\label{eq:4.2}
i\pt_t u +\Delta u+\lambda u\log(|u|^2)=0\quad\text{in}~H^{-1}(\omega)
%\quad\text{for all}~\omega\subset\subset\Omega.
\end{align}
for all $\omega\subset\subset\Omega$ and all $t\in\R$, and with $u(0)=\varphi$.
If in addition we assume $\varphi\in W_1$, the $H^1$-solution above satisfies $u\in (C\cap L^\infty)(\R,W_1)$ if $\lambda>0$ and $u\in C(\R,W_1)$ if $\lambda<0$. Moreover, the $W_1$-solution $u$ satisfies the equation \eqref{eq:4.2} in the sense of $W_1^*$.
%, where $W_1^*$ is the dual space of $W_1$.
\end{theorem}
%%%%%%%%%%%%%%%%%%

When we consider \eqref{eq:4.1} in $H^2$ setting, the Laplacian $\Delta$ is understood to be the self-adjoint realization in the Hilbert space $L^2(\Omega)$ with the domain 
\begin{align*}
D(\Delta)=\left\{ \varphi\in H^1_0(\Omega) :\Delta\varphi\in L^2(\Omega) \r\}.
\end{align*}
If $\Omega$ has a bounded boundary of class $C^2$, then $D(\Delta)=H^2(\Omega)\cap H^1_0(\Omega)$.
% (see \cite[Remark 2.6.3]{CH98}).
The $H^2$-energy space for \eqref{eq:4.1} is defined by
\begin{align*}
W_2=W_2(\Omega)=\left\{ \varphi\in D(\Delta) :\varphi\log(|\varphi|^2)\in L^2(\Omega) \r\}.
\end{align*}
Then Theorem \ref{thm:1.2} can be restated for \eqref{eq:4.1} as follows.
\begin{theorem}
\label{thm:4.2}
Let $\lambda\in\R\setminus\{0\}$ and let $\Omega$ be a general domain. For any $\varphi\in W_2$, there exists a unique solution 
%of \eqref{eq:1.1} in the class of 
\begin{gather*}
u\in W^{1,\infty}_{\rm loc}(\R, L^2(\Omega)),~
\pt_t u\in C_w(\R, L^2_{\rm loc}(\Omega)),
\\
\Delta u \in (C_w\cap L^\infty_{\rm loc})(\R, L^2_{\rm loc}(\Omega)),
%\quad \forall \omega\subset\subset\Omega
\end{gather*}
to \eqref{eq:4.1} in the sense of
\begin{align}
\label{eq:4.3}
i\pt_t u +\Delta u+\lambda u\log(|u|^2)=0\quad\text{in}~L^2(\omega)
%\quad\text{for all}~\omega\subset\subset\Omega.
\end{align}
for all $\omega\subset\subset\Omega$ and a.e. $t\in\R$, with $u(0)=\varphi$.
If we assume $\lambda>0$ and $\Omega$ has a bounded boundary of class $C^2$, then $u\in C(\R, W_2)$ and \eqref{eq:4.3} holds in $L^2(\Omega)$ and for all $t\in\R$.
%with $\omega=\Omega$
%by replacing $L^2(\omega)$ with $L^2(\Omega)$.
\end{theorem}
%%%%%%%%
In the proof of Theorem \ref{thm:4.2}, we use the following interpolation inequality instead of \eqref{eq:3.4}:
\begin{align*}
\norm[\nabla f]_{L^{2p'}}\cleq \norm[f]_{H^2}^{\frac{1}{2}+\frac{d}{4p}} \norm[f]_{L^2}^{\frac{1}{2}-\frac{d}{4p}}
\cleq  \left(\norm[f]_{L^2}+\norm[\Delta f]_{L^2}\r)^{\frac{1}{2}+\frac{d}{4p}} \norm[f]_{L^2}^{\frac{1}{2}-\frac{d}{4p}},
\end{align*}
where we use the elliptic regularity (see \cite[Theorem 9.25]{Bre11}) and need to impose some regularity properties of $\Omega$.\footnote{The regularity assumption on $\Omega$ in Theorem \ref{thm:4.2} may be weakened, but we do not purse this further here.} 
The boundary condition of $\Omega$ is used only in this step, and the rest of the proof is no different from the one of Theorem \ref{thm:1.2}.

%%%%%%%%%%%%%%
%%%%%%%%%%%%%%

\appendix
\section{Logarithmic inequalities}
\label{sec:A}
In this section we prove logarithmic inequalities used in the body of the paper.
%in the proof of main theorems.
\begin{lemma}
\label{lem:A.1}
For all $u,v\in\C$ and $\eps,\mu>0$, we have
\begin{align}
\label{eq:A.1}
\bigl|\Im \bigl(u\log (|u|+\eps)-v\log(|v|+\mu) \bigr) (\overline{u}-\overline{v}) \bigr|
\le |u-v|^2+\abs[\eps-\mu]\abs[u-v].
\end{align}
\end{lemma}
%%%%%%%
%%%%%%%
\begin{proof}
We first note that
\begin{align*}
\Im\bigl[u\log (|u|+\eps)-v\log(|v|+\mu) (\overline{u}-\overline{v})\bigr]
=\Im(u\overline{v})\bigl(\log(|u|+\eps)-\log(|v|+\mu) \bigr).
\end{align*}
Then using the inequalities
\begin{align*}
\Im (u\overline{v})\le \min\{|u|, |v|\}\abs[u-v]
\end{align*}
and
\begin{align*}
\left|\log (|u|+\eps)-\log(|v|+\mu)\r|
&=\left|\int_0^1 \frac{d}{dt} \log\bigl( t(|u|+\eps)+(1-t)(|v|+\mu) \bigr) dt\r|
\\
&=\left|\int_0^1 \frac{|u|+\eps-(|v|+\mu)}{t|u|+(1-t)|v|+t\eps+(1-t)\mu}dt\r|
\\
&\leq \frac{1}{\min\{|u|,|v|\}} \left(|u-v|+|\eps-\mu|\r), 
\end{align*}
we obtain \eqref{eq:A.1}.
%This completes the proof.
\end{proof}
%%%%%%%%%%%%
%%%%%%%%%%%%
\begin{lemma}
\label{lem:A.2}
For $\alpha\in (0,1)$ there exists $C(\alpha){>0}$ such that for all $u,v\in\C$, $\eps\in(0,1)$
\begin{align}
\label{eq:A.2}
&\abs[ {v\log(|v|+\eps)-u\log\abs[u]} ]
\le
\begin{aligned}[t]
&\eps+\abs[u-v]+C(\alpha)\times
\\&{}\bigl(1+|u|^{1-\alpha}\log^+\abs[u]+|v|^{1-\alpha}\log^+\abs[v] \bigr)
\abs[u-v]^\alpha,
\end{aligned}
\end{align}
where $\log^+x\ce\max\{\log x , 0\}$.
\end{lemma}
%%%%%%
\begin{proof}
Similarly to the proof of Lemma \ref{lem:A.1}, we obtain
\begin{align}
\label{eq:A.3}
\begin{aligned}
\abs[ \log(|v|+\eps)-\log|u| ]
&=\left| \int_0^1\frac{d}{dt}\log\Bigl( t(|v|+\eps)+(1-t)|u|\Bigr) dt \r|
\\
&=\left| \int_0^1 \frac{(|v|+\eps)-|u|}{ t(|v|+\eps)+(1-t)|u|}dt \r|
\le \frac{|u-v|+\eps}{ \min\{|u|, |v|\} }.
\end{aligned}
\end{align}
%%%%%%%%%
We now distinguish two cases.
\\[3pt]
{\bf Case 1: $|v|\le|u|$}. We note that
\begin{align*}
v\log(|v|+\eps)-u\log|u|
=v\left(\log(|v|+\eps)-\log|u|\r)+(v-u)\log|u|.
\end{align*}
Using \eqref{eq:A.3}, we estimate the first term on the right-hand side as
\begin{align*}
\abs[ v\left(\log(|v|+\eps)-\log|u|\r) ]\le |v|\cdot \frac{|u-v|+\eps}{ |v| }
=|u-v|+\eps
\end{align*}
and the second term as
\begin{align*}
\abs[ (v-u)\log|u|] &\le |v-u|^\alpha\abs[v-u]^{1-\alpha}\log |u|
\\
&\cleq|v-u|^\alpha\abs[u]^{1-\alpha}\log|u|
\cleq |v-u|^\alpha(1+\abs[u]^{1-\alpha}\log^+|u|),
\end{align*}
where we have used that $(0,1)\ni s\mapsto s^{1-\alpha}\log s$ is bounded. Therefore, we have
\begin{align*}
\abs[v\log(|v|+\eps)-u\log|u| ]\cleq |u-v|+\eps +  |v-u|^\alpha(1+\abs[u]^{1-\alpha}\log^+|u|).
\end{align*}
\noindent{\bf Case 2: $|u|\le|v|$}. In this case we write
\begin{align*}
v\log(|v|+\eps)-u\log|u|
=u\left( \log(|v|+\eps)-\log|u|\r)+\log(|v|+\eps)(v-u).
\end{align*}
Similarly to Case 1, we obtain
\begin{align*}
\abs[ u\left( \log(|v|+\eps)-\log|u|\r) ] &\le |u-v|+\eps,
\\[3pt]
\abs[ \log(|v|+\eps)(u-v)]&\cleq \abs[\log(|v|+\eps)]|v|^{1-\alpha}\abs[u-v]^{\alpha}
\\
&\cleq (1+|v|^{1-\alpha}\log^+|v|)\abs[u-v]^{\alpha},
\end{align*}
where the inequalities hold independent of $\eps\in(0,1)$.

Gathering the estimates in Case 1 and Case 2, we obtain \eqref{eq:A.2}.
\end{proof}
%%%%%%%
%%%%%%%
%\begin{remark}
%\eqref{eq:A.2}の評価は$(0,1)\ni s\mapsto s\log s$が$\alpha$次H\"older連続であることが効いている．
%\end{remark}
%%%%%

\section{The energy spaces $W_1$ and $W_2$}
%Inclusion relation in function spaces
\label{sec:B}%
\subsection{Characterization in terms of Orlicz spaces}
%%label{sec:Orl}%
In this subsection we see that the sets $W_1$ and $W_2$ defined by \eqref{eq:1.6} and \eqref{eq:1.7} respectively are characterized in terms of Orlicz spaces. By following \cite{C83}, we organize the bare minimum of basic properties of $W_1$ and $W_2$ relevant to this paper. We refer to \cite{KR61} for more information about Orlicz spaces.

We define the functions $F_1$, $F_2$, $A_1$, and $A_2$ on $\R^+$ by
\begin{align*}
F_1(s)&={-}s^2\log(s)^2,\quad F_2(s)=s^2\bigl(\log (s^2)\bigr)^2,
\\
A_1(s)&=
\left\{
\begin{aligned}
&{-}s^2\log (s^2)&&\text{if}~0\le s\le e^{-3},
\\
&3s^2+4e^{-3}s-e^{-6}&&\text{if}~s\ge e^{-3},
\end{aligned}
\r.
\\
A_2(s)&=
\left\{
\begin{aligned}
&s^2\bigl(\log (s^2)\bigr)^2&&\text{if}~0\le s\le e^{-3},
\\
&4s^2+40e^{-3}s-8e^{-6}&&\text{if}~s\ge e^{-3}.
\end{aligned}
\r.
\end{align*}
We note that both $A_1$ and $A_2$ are convex $C^1$-functions, which are $C^2$ and positive except at origin. We define the set $X_m$ by
\begin{align*}
X_m&=\left\{ \varphi\in L^1_{\rm loc}(\R^d) :A_m(|\varphi|)\in L^1(\R^d)\r\},
\end{align*}
and the Luxembourg norm by
\begin{align*}
\norm[\varphi]_{X_m}=\inf\left\{ k>0 : \int A_m\left(\frac{|\varphi|}{k}\r)\le 1 \r\}
\end{align*}
for $m=1,2$.
\begin{lemma}%[{\cite[Lemma 2.1]{C83}}]
\label{lem:B.1}
For $m=1,2$, the following properties hold:
\begin{enumerate}
\item $(X_m, \norm[\,\cdot\,]_{X_m})$ is a Banach space.
\label{enu:1}

\item If $\varphi_n\to\varphi$ in $X_m$, then $A_m(|\varphi_n|)\to A_m(|\varphi|)$ in $L^1(\R^d)$.
\label{enu:2}

\item If $\varphi_n\to\varphi$ a.e. and if 
\begin{align*}
\int A_m(|\varphi_n|) \to \int A_m(|\varphi|), 
\end{align*}
then $\varphi_n\to\varphi$ in $X_m$.
\label{enu:3}
\end{enumerate}
\end{lemma}
%%%%
\begin{proof}
The proof for $m=1$ is given in \cite[Lemma 2.1]{C83}. The same proof still holds for $m=2$.
\end{proof}
%%%%
The sets $W_1$ and $W_2$ are represented by $X_1$ and $X_2$ as follows.
%%%%%%
\begin{lemma}
\label{lem:B.2}
For $m=1,2$, $W_m=H^m\cap X_m$ and then $W_m$ is a Banach space equipped with the norm 
$\norm[\varphi]_{W_m}\ce\norm[\varphi]_{H^m}+\norm[\varphi]_{X_m}$.
\end{lemma}
%%%
\begin{proof}
We follow the argument of \cite[Proposition 2.2]{C83}. We set 
\begin{align*}
B_m(s)=F_m(s)-A_m(s)
\end{align*}
for $m=1,2$. For any small $\delta>0$ 
we have
\begin{align*}
\int \abs[B_m(|\varphi|)] \cleq \int |\varphi|^{2+\delta},
%\le C(\delta)
\end{align*}
which yields that $B_m(|\varphi|)\in L^1(\R^d)$ if $\varphi\in H^m(\R^d)$. Therefore we deduce that
\begin{align}
\label{eq:B.1}
\text{if $\varphi\in H^m(\R^d)$,}\quad
F_m(|\varphi|)\in L^1(\R^d){\iff} A_m(|\varphi|)\in L^1(\R^d).
\end{align}
Hence the result follows from Lemma \ref{lem:B.1} \ref{enu:1}.
\end{proof}
%%%%%%%%%%%
%%%%%%%%%%%
\begin{lemma}
\label{lem:B.3}
For $m=1,2$, if $u\in C(\R, H^m(\R^d))$, then
\begin{align*}
t\mapsto \int F_m(|u(t)|)\in C(\R) \iff u\in C(\R,W_m).
\end{align*}
\end{lemma}
%%%%%%%%%
\begin{proof}
%Similarly as in the proof of Lemma \ref{lem:B.2},
We first note that for any small $\delta>0$
% and for $u,v\in\C$,
\begin{align*}
\abs[ B_m(|u|)-B_m(|v|) ] \cleq (|u|+|v|)^{1+\delta}\abs[u-v]
\quad\text{for all}~u,v\in\C.
\end{align*}
This yields that if $u\in C(\R,H^m(\R^d))$, then 
\begin{align*}
t\mapsto \int B_m(|u(t)|)\in C(\R).
\end{align*}
Therefore, if $u\in C(\R, H^m(\R^d))$, then
\begin{align*}
t\mapsto \int F_m(|u(t)|)\in C(\R)\iff 
t\mapsto \int A_m(|u(t)| ) \in C(\R).
\end{align*}
Hence the result follows from 
Lemma \ref{lem:B.1} \ref{enu:2}, \ref{enu:3}, and Lemma \ref{lem:B.2}.
\end{proof}

\subsection{Inclusion relation in function spaces}
In this subsection we study the inclusion relation \eqref{eq:1.8} in function spaces.
%for $\alpha\in(0,1]$.
%%が成り立ち，何れも包含関係は真になっている．ここではこの事実を簡単な例でチェックしてみる．
%%便宜上，補助空間として
%For convenience we prepare
%\begin{align*}
%X_1&\ce \left\{ \varphi\in L^2(\R^d) : |\varphi|^2\log(|\varphi|^2)\in L^1(\R^d) \r\},
%\\
%X_2&\ce\left\{ \varphi\in L^2(\R^d) : \varphi\log(|\varphi|^2) \in L^2(\R^d) \r\}.
%\end{align*}
%We note that $W_1$ and $W_2$ are rewritten as
%\begin{align*}
%W_1=H^1(\R^d)\cap X_1,\quad W_2=H^2(\R^d)\cap X_2.
%\end{align*}
%%
%We prove the following result.
First we have the following result.
\begin{lemma}
\label{lem:B.4}
%The following relation holds:
For $m=1,2$, and $\alpha\in(0,1]$, we have
\begin{align}
\label{eq:B.2}
H^m\cap\scF(H^\alpha)\subset W_m\subset H^m(\R^d).
\end{align}
\end{lemma}
%%%%%%%%%%%%%%%%
%%%%%%%%%%%%%%%%
\begin{proof}
For $\varphi\in H^m\cap\scF(H^\alpha)$, we have
\begin{align*}
\int |F_m(\varphi)|\cleq  \int|\varphi|^{2-\delta}+\int|\varphi|^{2+\delta}\quad\text{for any small $\delta>0$}.
\end{align*}
The first term on the right-hand side is controlled by the inequality
\begin{align*}
\int|\varphi|^{2-\delta}\cleq \norm[\varphi]_{L^2}^{2-\delta-\frac{d\delta}{2\alpha}}\norm[x^\alpha \varphi]_{L^2}^{\frac{d\delta}{2\alpha}},
\end{align*}
see, e.g., (2.3) in \cite{CG18}. Therefore, along with Sobolev's embedding, we deduce that
\begin{align*}
\int |F_m(\varphi)|\cleq \norm[\braket{x}^\alpha \varphi]_{L^2}^{2-\delta}+\norm[\varphi]_{H^m}^{2+\delta},
\end{align*} 
which implies the claim.
\end{proof}
%%%%%%%%%%%%
%%%%%%%%%%%%
Next we prove the following result.
\begin{lemma}
\label{lem:B.5}
%The following properties hold.
\begin{align}
\label{eq:B.3}
\braket{x}^{-d/2}\left(\log\braket{x}\r)^{-\rho}
\in
\left\{
\begin{aligned}
&L^1_{\rm loc}(\R^d)\setminus L^2(\R^d)&&\text{if}~0<\rho\le1/2,
\\
&H^2(\R^d)\setminus W_1&&\text{if}~1/2<\rho\le1,
%\rho\in(1/2,1],
\\
&(H^2(\R^d)\cap W_1)\setminus W_2&&\text{if}~1<\rho\le3/2,
%~\rho\in(1,3/2],
\\
&W_2\setminus \bigcup_{\alpha\in(0,1]}\scF(H^\alpha)&&\text{if}~\rho>3/2.
%~\rho\in(3/2,\infty).
\end{aligned}
\r.
\end{align}
In particular each inclusion in \eqref{eq:B.2} is strict.
\end{lemma}
%%%%%%%%%%%%%
%%%%%%%%%%%%%
\begin{proof}
We set
\begin{align*}
f_\rho (x)=\braket{x}^{-d/2}\left(\log\braket{x}\r)^{-\rho}\quad\text{for}~\rho>0.
\end{align*}
We note that
\begin{align*}
\int |f_\rho|^2 \sim \int_2^\infty \frac{dr}{r(\log r)^{2\rho}}
\left\{
\begin{aligned}
&=\infty&&\text{if}~0<\rho\le1/2,
%~\rho\in(0,1/2],
\\
&<\infty&&\text{if}~\rho>1/2,
\end{aligned}
\r.
\end{align*}
which yields that
\begin{align}
\label{eq:B.4}
f_\rho \in
\left\{ 
\begin{aligned}
&L^1_{\rm loc}(\R^d)\setminus L^2(\R^d) &&\text{if}~0<\rho\le1/2,
\\
&L^2(\R^d)&&\text{if}~\rho>1/2.
\end{aligned}
\r.
\end{align}

One can easily verify that $\nabla f_\rho,\,\Delta f_\rho\in L^2(\R^d)$ for any $\rho>0$, so we obtain that $f_\rho\in H^2(\R^d)$ if $\rho>1/2$.
Combining this with \eqref{eq:B.1}, we deduce that
\begin{align}
\label{eq:B.5}
\text{if}~\rho>1/2,\quad f_\rho\in X_m {\iff} F_m(|f_\rho|)\in L^1(\R^d)
\end{align}
for $m=1,2$.
A direct calculation shows that
\begin{align*}
|f_\rho|^2\abs[\log(|f_\rho|^2) ]
 &=\braket{x}^{-d}\left(\log\braket{x}\r)^{-2\rho}
\abs[ -d\log\braket{x}-2\rho\log\log\braket{x} ]
\\
&\sim \braket{x}^{-d}\left(\log\braket{x}\r)^{-2\rho+1}\quad\text{as}~|x|\to\infty,
\end{align*}
which implies that
\begin{align*}
\int |f_\rho|^2\abs[\log(|f_\rho|^2)] \sim\int_{2}^\infty \frac{dr}{r(\log r)^{2\rho -1}}
\left\{
\begin{aligned}
&=\infty&&\text{if}~0<\rho\le1,
\\
&<\infty&&\text{if}~\rho>1.
\end{aligned}
\r.
\end{align*}
Similarly, we have
\begin{align*}
|f_\rho|^2\left(\log|f_\rho|^2\r)^2
&=\braket{x}^{-d} ( \log\braket{x})^{-2\rho}
\bigl( -d\log\braket{x}-2\rho\log\log\braket{x}\bigr)^2
\\
&\sim \braket{x}^{-d}\left(\log\braket{x}\r)^{-2\rho+2}\quad
\text{as}~|x|\to\infty,
\end{align*}
which implies that
\begin{align*}
\int |f_\rho|^2\left(\log(|f_\rho|^2) \r)^2\sim\int_{2}^\infty \frac{dr}{r(\log r)^{2\rho-2}}
\left\{
\begin{aligned}
&=\infty&&\text{if}~0<\rho\le3/2,
\\
&<\infty&&\text{if}~\rho>3/2.
\end{aligned}
\r.
\end{align*}
Hence, it follows from \eqref{eq:B.5} that
\begin{align}
\label{eq:B.6}
f_\rho\in
\left\{ 
\begin{aligned}
&L^2(\R^d)\setminus X_1 &&\text{if}~1/2<\rho\le1,
\\
&X_1\setminus X_2&&\text{if}~1<\rho\le3/2,
\\
&X_2&&\text{if}~\rho>3/2.
\end{aligned}
\r.
\end{align}

Finally we check the relation between $\scF(H^\alpha)$ and $X_2$. 
For any small $\alpha\in(0,1]$, we have
\begin{align*}
\braket{x}^{2\alpha}|f_\rho|^2=\braket{x}^{-d+2\alpha}\left(\log\braket{x}\r)^{-2\rho}\cgeq \braket{x}^{-d+\alpha},
\end{align*}
which implies that $f_\rho\notin \scF(H^\alpha)$. 
%no matter how small is $\alpha\in(0,1]$. 
Therefore, we obtain
\begin{align}
\label{eq:B.7}
f_\rho \in X_2\setminus\bigcup_{\alpha\in(0,1]}\scF(H^\alpha)
\quad\text{if}~\rho>3/2.
\end{align}
Hence, the result follows from \eqref{eq:B.4}, \eqref{eq:B.6}, \eqref{eq:B.7}, and Lemma \ref{lem:B.2}.
%This completes the proof of \eqref{eq:B.3}.
%%
%The last claim in Lemma \ref{lem:B.5} follows from \eqref{eq:B.3} and Lemma \ref{lem:B.2}.
\end{proof}
%%%%%%%

\section*{Acknowledgments}

M.H. is supported by JSPS KAKENHI Grant Number JP22K20337 and by the Italian MIUR PRIN project 2020XB3EFL. T.O. is supported by JSPS KAKENHI Grant Numbers 18KK073 and 19H00644.

%\bibliographystyle{amsplain_abbrev_nobysame_nonumber}
%\bibliography{bib-logNLS}

\end{document}